\documentclass[10pt]{amsart}
\usepackage{latexsym, amssymb, amsfonts, amscd}

\theoremstyle{plain}
\newtheorem{Thm}{Theorem}[subsection]
\newtheorem{Thm*}{Theorem}[section]
\newtheorem{Cor}[Thm]{Corollary}

\newtheorem{Prop}[Thm]{Proposition}
\newtheorem{Lem}[Thm]{Lemma}
\newtheorem{Cl}[Thm]{Claim}

\theoremstyle{definition}
\newtheorem{Def}[Thm]{Definition}

\newtheorem{Emp}[Thm]{}

\newtheorem{Not}[Thm]{Notation}

\numberwithin{equation}{section}

\newcommand{\qlbar}{\overline{\B{Q}}_{\ell}}

\newcommand{\La}{\Lambda}

\newcommand{\ov}{\overline}

\newcommand{\B}[1]{\mathbb#1}
\newcommand{\cal}[1]{\mathcal{#1}}
\newcommand{\C}[1]{\cal#1}

\newcommand{\isom}{\overset {\thicksim}{\to}}

\newcommand{\form}[1]{(\ref{Eq:#1})}

\newcommand{\si}{\sigma}

\newcommand{\lra}{\longrightarrow}

\newcommand{\hra}{\hookrightarrow}
\newcommand{\wt}{\widetilde}

\newcommand{\Gm}{\Gamma}

\newcommand{\gm}{\gamma}
\newcommand{\dt}{\delta}
\newcommand{\Dt}{\Delta}
\newcommand{\bs}{\backslash}

\newcommand{\lan}{\langle}
\newcommand{\ran}{\rangle}

\newcommand{\al}{\alpha}

\newcommand{\la}{\lambda}

\newcommand{\rs}[1]{Section \ref{S:#1}}
\newcommand{\rl}[1]{Lemma \ref{L:#1}}
\newcommand{\rn}[1]{Notation \ref{N:#1}}
\newcommand{\rcl}[1]{Claim \ref{C:#1}}
\newcommand{\rp}[1]{Proposition \ref{P:#1}}

\newcommand{\re}[1]{\ref{E:#1}}
\newcommand{\rco}[1]{Corollary \ref{C:#1}}

\newcommand{\rt}[1] {Theorem \ref{T:#1}}
\newcommand{\rd}[1]{Definition \ref{D:#1}}
\newcommand{\sm}{\smallsetminus}

\newcommand{\on}{\operatorname}

\newcommand{\pr}{\operatorname{pr}}

\newcommand{\Ker}{\operatorname{Ker}}

\newcommand{\val}{\operatorname{val}}

\newcommand{\Spec}{\operatorname{Spec}}
\newcommand{\Aut}{\operatorname{Aut}}
\newcommand{\Gr}{\operatorname{Gr}}

\newcommand{\Ad}{\operatorname{Ad}}
\newcommand{\colim}{\operatorname{colim}}

\newcommand{\Tr}{\operatorname{Tr}}
\newcommand{\Fr}{\operatorname{Fr}}

\newcommand{\Lie}{\operatorname{Lie}}
\newcommand{\Fl}{\operatorname{Fl}}

\renewcommand{\sc}{\operatorname{sc}}

\newcommand{\pt}{\operatorname{pt}}
\newcommand{\Ext}{\operatorname{Ext}}
\newcommand{\Hom}{\operatorname{Hom}}

\newcommand{\Span}{\operatorname{Span}}
\begin{document}

\title[Homology of affine Springer fibers]%
{Semi-infinite orbits in affine flag varieties and homology of
affine Springer fibers}
\author{Roman Bezrukavnikov}
\address{Department of Mathematics\\
Massachusetts Institute of Technology\\
77 Massachusetts Avenue\\
Cambridge, MA 02139, USA}
\email{bezrukav@math.mit.edu}
\date{\today}
\author{Yakov Varshavsky}
\address{Einstein Institute of Mathematics\\
Edmond J. Safra Campus\\
The Hebrew University of Jerusalem\\
Givat Ram, Jerusalem, 9190401, Israel}
\email{yakov.varshavsky@mail.huji.ac.il}

\thanks{This research was partially supported by
the BFS grants 2016363 and 2020189. The research of R.B. was also partially supported by NSF grant
DMS-1601953 and by grants from the Institute for Advanced Study
and Carnegie Corporation of New York. The research of Y.V. was partially supported by the
ISF grants 822/17 and 2091/21.}
\begin{abstract}
Let $G$ be a connected reductive group over an
algebraically closed field $k$, and let $\Fl$ be the affine flag variety of $G$. For every regular semisimple element $\gm$ of $G(k((t)))$, the affine Springer fiber $\Fl_\gm$ can be presented as a union of closed subvarieties $\Fl^{\leq w}_{\gm}$,
defined as the intersection of $\Fl_{\gm}$ with an affine Schubert variety
$\Fl^{\leq w}$.


\smallskip

The main result of this paper asserts that
if elements $w_1,\ldots,w_n$ are sufficiently regular, then
the natural map $H_i(\bigcup_{j=1}^n
\Fl^{\leq w_j}_{\gm})\to H_i(\Fl_{\gm})$ is injective for every
$i\in\B{Z}$. It plays an important role in our work
\cite{BV}. One can view this statement as providing a categorification
of the notion of a weighted orbital integral. Along the way we also show that every affine Schubert variety
can be written as an intersection of closures of semi-infinite orbits.
\end{abstract}
\maketitle


\tableofcontents

\section*{Introduction}

 Let $k$ be an algebraically closed field, $K:=k((t))$ the
field of Laurent power series over $k$, and
$\C{O}=\C{O}_K=k[[t]]$ the ring of integers of $K$. Let $G$ be a
connected reductive group over $k$, and let $G^{\sc}$ be the simply-connected covering of
the derived group of $G$. For an algebraic group $H$ over $K$ (resp. $\C{O}$), we denote by $LH$ (resp. $L^+(H)$)
the corresponding loop (resp. arc) group.

\smallskip

We fix a maximal torus $T\subseteq G$ and an Iwahori subgroup scheme
$I\subseteq L^+(G)$ such that $I\cap LT=L^+(T)$, and let $T_{G^{\sc}}\subseteq G^{\sc}$
and $I^{\sc}\subseteq L^+(G^{\sc})$ be the corresponding maximal torus and the
Iwahori subgroups of $G^{\sc}$, respectively. Let $W=W_G$ be the Weyl group of $G$, let $\La=X_*(T_{G^{\sc}})$ be the group of cocharacters, and let $\wt{W}:=W\ltimes\La$ be the affine Weyl group of $G$.

\smallskip

Denote by $\Fl=L(G^{\sc})/I^{\sc}$ the affine flag variety of $G^{\sc}$. Then we have a natural embedding $\wt{W}\hra \Fl$. For
every $w\in\wt{W}$, we denote by $\Fl^{\leq w}\subseteq \Fl$ the
closure of the $I^{\sc}$-orbit $I^{\sc}w\subseteq \Fl$. Then each $\Fl^{\leq w}$
is a closed projective subscheme of $\Fl$, usually
referred to as the affine Schubert variety, while $\Fl$ is an inductive limit of
$\Fl^{\leq w}$.

\smallskip

For a regular semi-simple element $\gm\in G(K)$ we denote by $\Fl_{\gm}\subseteq \Fl$
the corresponding affine Springer fiber, i.e., the closed ind-subscheme of points $gI^{\sc}\in\Fl$ such that $g^{-1}\gm g\in I$.

\smallskip

Let $G_{\gm}$ be the centralizer of $\gm$ in $G$. It is a torus defined over $K$, and let
$S_{\gm}\subseteq G_{\gm}$ be the maximal $K$-split torus. We will always assume that $S_{\gm}$ is contained in $T_K$,
where $T_K$ denote the extension of scalars of $T$ to $K$.

\smallskip

For every ind-subscheme $Z\subseteq \Fl_G$, we denote by $Z_{\gm}$ the intersection
$Z\cap \Fl_{\gm}$. Then $\Fl_{\gm}$ is a union of the $\Fl^{\leq w}_{\gm}$, hence each homology group $H_i(\Fl_{\gm})$ is by
definition  the direct limit of $H_i(\Fl_{\gm}^{\leq w})$.
The main result of this paper implies that the canonical map
$H_i(\Fl_{\gm}^{\leq w})\to H_i(\Fl_{\gm})$ is injective if $w$ is
sufficiently regular.

\smallskip

More precisely, let $\pi:\wt{W}\to\wt{W}/W=\La$ be the projection.
For $m\in\B{N}$ we say that $w\in\wt{W}$ is {\em
$m$-regular}, if $|\langle\al,\pi(w)\rangle|\geq m$ for every root $\al$
of $(G,T)$. The main goal of this paper is to prove the following result
used in our companion work \cite{BV}.

\begin{Thm*} \label{T:inj}
 There exists $m\in\B{N}$ (depending on $\gm$) such that for every finite set
$w_1,\ldots,w_n$ of $m$-regular elements of $\wt{W}$ the natural map
$H_i(\bigcup_{j=1}^n \Fl^{\leq w_j}_{\gm})\to H_i(\Fl_{\gm})$ is
injective for every $i\in\B{Z}$.
\end{Thm*}

If the group $G$ and element $\gm$ are defined over ${\mathbb F}_q$, the expression
\[
\big| (\bigcup_{j=1}^n \Fl^{\leq w_j}_{\gm})({\mathbb F}_q)\big|=\Tr(\Fr,H_*(\bigcup_{j=1}^n \Fl^{\leq w_j}_{\gm}))
\]
appears in computation of orbital integrals.

\smallskip

Theorem \ref{T:inj} allows one to interpret $H_*(\bigcup_{j=1}^n \Fl^{\leq w_j}_{\gm})$
as a term of a filtration on $H_*(\Fl_{\gm})$ which turns out to have favorable properties
with respect to the affine Springer action: it is a good filtration compatible with a natural
filtration on the group ring of the affine Weyl group. This provides a way to prove an expression for the
weighted orbital integral (or rather the closely related value of the averaging of a distribution)
in terms of  $H_*(\Fl_{\gm})$ equipped with an action an action of Frobenius and affine Springer
action; this application is developed in \cite{BV}. Thus Theorem \ref{T:inj} yields a categorification of the weighted orbital integral.

\smallskip

Theorem \ref{T:inj} will be deduced from a more general result. For each Borel subgroup $B\supseteq T$ of $G$, we denote
its unipotent radical by $U_{B}\subseteq G$. For every $w\in\wt{W}$,
we denote by $\Fl^{\leq_{B}w}\subseteq \Fl$ the closure of the
$U_{B}(K)$-orbit $U_{B}(K)w\subseteq \Fl$, which is called {\em the
semi-infinite orbit}. Then $\Fl^{\leq_{B}w}$ is a closed
ind-subscheme of $\Fl$.

\smallskip

We consider tuples $\ov{w}=\{w_B\}_B$ of elements of $\wt{W}$, where
$B$ runs over the set of all Borel subgroups  $B\supseteq T$ of $G$.
Most of the time will restrict ourselves to tuples, which are
{\em admissible} (see \rd{adm}) and {\em $m$-regular} (see \rn{kreg}).
In particular, the last assumption implies that each $w_B$ is $m$-regular.

\smallskip

For each tuple $\ov{w}$, we denote by $\Fl^{\leq\ov{w}}$
the reduced intersection $\bigcap_{B}\Fl_G^{\leq_{B}w_{B}}$. Each $\Fl^{\leq\ov{w}}$
is a projective scheme (see \rco{adm}(c)).

\smallskip

\rt{inj} follows from the following two results:

\begin{Thm*} \label{T:tuple}
For every $w\in\wt{W}$, there exists a unique admissible tuple
$\ov{w}$  such that $\Fl^{\leq w}=\Fl^{\leq\ov{w}}$.
Moreover, there exists $r\in\B{N}$ such that
for every $m\in\B{N}$ and every $(m+r)$-regular $w\in\wt{W}$, the
tuple $\ov{w}$ is $m$-regular.
\end{Thm*}

\begin{Thm*} \label{T:inj'}
 There exists $m\in\B{N}$ (depending on $\gm$) such that for every finite set
 $\ov{w}_1,\ldots,\ov{w}_n$ of $m$-regular admissible
tuples, the natural map
$H_i(\bigcup_{j=1}^n \Fl^{\leq \ov{w}_j}_{\gm})\to H_i(\Fl_{\gm})$ is
injective for all $i$.
\end{Thm*}

Notice that \rt{inj'} is essentially vacuous, if $\gm$ is elliptic. Indeed, in this case the affine Springer fiber
$\Fl_{\gm}$ is of finite type, so there exists an integer $m$ such that for every $m$-regular admissible tuple $\ov{w}$ we have an equality
$\Fl_{\gm}^{\leq\ov{w}}=\Fl_{\gm}$.

\smallskip

 To show the assertion in general, we use induction on the semisimple rank of $G$. Namely, if $\gm$ is not elliptic, then $\Fl_{\gm}$ is equipped with an action of a nontrivial torus $S$, and the scheme of
fixed points $\Fl_{\gm}^S$ is naturally isomorphic to a disjoint union of affine Springer fibers corresponding to a proper Levi subgroup $M$
of $G$. Thus an analog of \rt{inj'} for $\Fl_{\gm}^S$ holds by induction hypothesis, and we use finiteness properties of $H_i(\Fl_{\gm})$ and
localization theorem in equivariant cohomology to relate homology of $\Fl_{\gm}$ with that of $\Fl_{\gm}^S$.

\smallskip

The paper is organized as follows. In Section 1 we study orderings on affine
Weyl groups and introduce admissible tuples, which play a central role later. In Section 2 we
study semi-infinite orbits in affine flag varieties and their intersections, establish \rt{tuple},
and show technical results needed later. In Section 3 we study geometric properties of the affine Springer fibers, and establish a finiteness property of its homology.

\smallskip

Finally, in Section 4 we prove \rt{inj'} using results of the previous sections.
Namely, we review the localization theorem in the equivariant cohomology with compact support in subsection 4.1, give a criterion of an injectivity of the map on homology in subsection 4.2, and complete the proof in subsection 4.3.



\section{Combinatorics of affine Weyl groups}

\subsection{Preliminaries}

\begin{Emp} \label{E:root}
{\bf Roots}.

\smallskip

(a) Let $V$ be a finite dimensional vector space
over $\B{R}$, $V^*$ the dual space, and let $\Phi\subseteq V^*$ be a
(reduced) root system (see, for example, \cite{Be} or \cite[Section~VI]{Bo}).

\smallskip

(b) We denote by $\C{C}=\C{C}_{\Phi}$ the set of all Weyl chambers
$C\subseteq V$ of $\Phi$.  For each $C\in\C{C}$, we denote by
$\Phi_{C}\subseteq\Phi$ the set of $C$-positive roots, by
$\Dt_{C}\subseteq\Phi_C$ the set of $C$-simple roots, and by
$\Psi_C\subseteq V^*$ the set of $C$-fundamental weights.

\smallskip

(c) Denote by $\wt{\Phi}\supseteq\Phi$ the set of affine roots. Then
$\wt{\Phi}$ is a collection of pairs $(\al,n)\in\Phi\times\B{Z}$,
and every $\wt{\al}=(\al,n)$ is identified with an affine
function $\wt{\al}:V\to\B{R}$, given by the rule
$\wt{\al}(x)=\al(x)+n$. For each subset $\Phi'\subseteq \Phi$ and
$\wt{\al}=(\al,n)\in\wt{\Phi}$, we will say that $\wt{\al}\in
\wt{\Phi'}$ if $\al\in\Phi'$.

\smallskip

(d) Let $W=W_{\Phi}\subseteq \Aut (V)$ be the Weyl group of $\Phi$,
let $\La\subseteq V$ be the subgroup generates by coroots
$\{\check{\al}\}_{\al\in\Phi}$, and let $\wt{W}:=W\ltimes\La$ be
the affine Weyl group of $\Phi$. We will denote by $\pi$ the
natural projection $\wt{W}\to\wt{W}/W=\La$.

\smallskip

(e) The lattice $\La$ acts on $V$ by translations. Then the group
$\wt{W}$ act on $V$ by affine transformations, hence it acts on
$\wt{\Phi}$ by the rule $w(\wt{\al})(x)=\wt{\al}(w^{-1}(x))$ for
all $x\in V$. In particular, for each $\mu\in\La$ and
$(\al,n)\in\wt{\Phi}$, we have $\mu(\al,n)=(\al,n-\langle\al,\mu\rangle)$.

\smallskip

(f) For each $\wt{\al}\in\wt{\Phi}$, the affine reflection $s_{\wt{\al}}$ satisfies
$s_{\wt{\al}}(x)=x-\wt{\al}(x)\check{\al}$ for all  $x\in V$. In particular, for all $(\al,n)\in\wt{\Phi}$,
we have equality $s_{\al,n}=(-n\check{\al})s_{\al}\in\wt{W}$.

\smallskip

(g) For each $\al\in\Phi$, we denote by
$\wt{W}_{\al}\subseteq\wt{W}$ the subgroup generated by reflections
$s_{\wt{\al}}$, with $\wt{\al}=(\al,n), n\in\B{Z}$.
\end{Emp}

\begin{Emp} \label{E:weyl}
{\bf The fundamental Weyl chamber.}

\smallskip

(a) We fix a Weyl chamber
$C_0\in\C{C}$, and denote by $A_0$ the fundamental alcove such
that $A_0\subseteq C_0$ and such that $0\in V$ lies in the closure
of $A_0$.

\smallskip

(b) The choice of $C_0$ defines the set of positive roots
$\Phi_{>0}=\Phi_{C_0}\subseteq \Phi$ and the set of positive affine
roots $\wt{\Phi}_{>0}\subseteq \wt{\Phi}$. Explicitly,
$\wt{\al}=(\al,n)\in\wt{\Phi}$ is positive if and only if either $n>0$, or
$n=0$ and $\al>0$.

\smallskip

(c) Then $C_0$ defines  a set of
simple reflection $S\subseteq W$, and $A_0$ defines a set of simple
affine reflections $\wt{S}\subseteq\wt{W}$. In particular, a choice
of $C_0$ define length functions and Bruhat orders $\leq$ on
both  $W$ and $\wt{W}$.

\smallskip

(d) Using $A_0$, we identify each $w\in\wt{W}$ with the corresponding alcove
$w(A_0)\subseteq V$. In particular, will say that $w\in\wt{W}$
belongs to $C\in\C{C}$, or $w\in C$, if $w(A_0)\subseteq C$.
Explicitly, this means that $\langle\al,w(A_0)\rangle =\langle
w^{-1}(\al),A_0\rangle\geq 0$ for each  $\al\in\Phi_{C}$,
or, what is the same, $w^{-1}(\Phi_{C})\subseteq \wt{\Phi}_{>0}$.
\end{Emp}

\begin{Emp} \label{E:fund}
{\bf Fundamental weights.}

\smallskip

(a) We set $\Psi:=\bigcup_{C\in\C{C}}\Psi_C\subseteq V^*$. For $\psi\in\Psi$ and $C\in\C{C}$, we write
$C\owns\psi$, if $\psi\in\Psi_C$.

\smallskip

(b) Every $\psi\in\Psi$ gives rise to a fundamental coweight $\check{\psi}\in \La_{\B{Q}}:=\La\otimes_{\B{Z}}\B{Q}\subseteq V$.
Namely, $\check{\psi}$ is characterized by condition that for every  $C\in\C{C}$ such that $\psi\in \Psi_C$ and every
$\al\in\Dt_C$ we have $\lan\al,\check{\psi}\ran=\lan\psi,\check{\al}\ran$.
In particular, for every $C\in\C{C}$ we have $\psi\in \Psi_C$ if and only if $\check{\psi}$ lies in
the closure of $C$.

\smallskip

(c) For every $\psi\in\Psi$, we denote by $\Phi(\psi)$ (resp.
$\Phi^{\psi}$) the set of $\al\in\Phi$ such that
$\langle\al,\check{\psi}\rangle\geq 0$ (resp.
$\langle\al,\check{\psi}\rangle=0$). Notice that $\Phi^{\psi}$ is
a root system, and there is a bijection $C\mapsto C^{\psi}$
between Weyl chambers $C\owns\psi$ of $\Phi$ and Weyl chambers of
$\Phi^{\psi}$. This bijection satisfies the property that
$(\Phi^{\psi})_{C^{\psi}}=\Phi_{C}\cap \Phi^{\psi}$. We denote by $\wt{W}^{\psi}\subseteq \wt{W}$
the Weyl affine group of $\Phi^{\psi}$.

\smallskip

(d) For every $\psi\in\Psi$, we fix a Weyl chamber $C^{\psi}_0$ of $\Psi^{\psi}$. As in Section~\re{weyl}(b), this choice
defines the set of positive affine roots $\wt{\Phi}^{\psi}_{>0}\subseteq \wt{\Phi}^{\psi}$, and we denote by
$\wt{W}_{\psi}\subseteq \wt{W}$ the set of all $w\in\wt{W}$ such that
$w^{-1}(\wt{\Phi}^{\psi}_{>0})\subseteq\wt{\Phi}_{>0}$. Then for every
$w\in\wt{W}$ there exists a unique decomposition
$w=w^{\psi}w_{\psi}$, where $w^{\psi}\in\wt{W}^{\psi}$ and
$w_{\psi}\in\wt{W}_{\psi}$ (compare, for example, \cite[Lemma~B.1.7(b)]{BV}).
In other words, $\wt{W}_{\psi}\subseteq\wt{W}$ is a set of representatives
of the set of left cosets $\wt{W}^{\psi}\bs\wt{W}$.
\end{Emp}




\begin{Emp} \label{E:border}
{\bf Properties of the Bruhat order.}

\smallskip

(a) Let $w',w''\in\wt{W}$ and $s\in\wt{S}$ be such that $w'\leq w''$. Then we
have either $w's\leq w''s$ (resp. $sw'\leq sw''$) or $w's\leq w''$
and $w'\leq w''s$ (resp. $sw'\leq w''$ and $w'\leq sw''$) or both
(see, for example, \cite[Proposition~2.2.7]{BB}).

\smallskip

(b) Let $w',w''\in\wt{W}$ and $s\in\wt{S}$ be such that $sw'<w'$ and
$sw''<w''$. Then, by part~(a), we have $w'\leq w''$ if and only if $sw'\leq sw''$.

\smallskip

(c) Let $w,w'$ and $w''$ be elements of $\wt{W}$ such that
$l(ww')=l(w)+l(w')$ and $ww'\leq ww''$. Then  $w'\leq w''$.
Indeed,  if $w=s\in\wt{S}$, then the assertion follows from part~(a). The
general case follows by induction on $l(w)$. By a similar argument, if
$l(ww'')=l(w)+l(w'')$ and $w'\leq w''$, then $ww'\leq ww''$.

\smallskip

(d)  For every $\mu\in\La$ and $u\in W$, we have $l(u\mu
u^{-1})=l(\mu)$. Indeed, it is enough to show the assertion in the case
$u=s=s_{\al}$ for a simple root $\al$.  In this case, we have $s\mu s=\mu$,
if $\langle\al,\mu\rangle=0$; $s\mu>\mu>\mu s$ if
$\langle\al,\mu\rangle>0$; and $s\mu<\mu<\mu s$ if
$\langle\al,\mu\rangle<0$.

\smallskip

(e) Note that $w\in\wt{W}$ belongs to  $C_0$ if and only if
$l(sw)>l(w)$ for every  $s\in S$. In other words,
$\wt{W}\cap C_0$ is the set of the shortest representatives of cosets
$W\bs\wt{W}$. In particular, for every $w\in \wt{W}\cap
C_0$ and $u\in W$ we have $l(uw)=l(u)+l(w)$ and for every $u\leq u'$ in $W$
we have $uw\leq u'w$.

\smallskip

(f) The characterization of $C_0$ given in part~(e) implies that for
every $w\in\wt{W}\cap C_0$ and  $s\in\wt{S}$ with $ws<w$, we
have $ws\in C_0$.

\smallskip

(g) For every $u\in W$ and every $\mu\in\La\cap C_0$ we have
$u\leq\mu$. Indeed, it is enough to show that $u\leq_R\mu$ (see
\cite[Definition~3.1.1]{BB}). Hence, by \cite[Proposition~3.1.3]{BB}, it is
enough to show that for every affine root $\wt{\al}>0$  such
that $u(\wt{\al})<0$ we have $\mu(\wt{\al})<0$. If
$\wt{\al}=(\al,n)>0$ satisfies $u(\wt{\al})=(u(\al),n)<0$, then
$n=0$, and $\al>0$. Hence
$\mu(\wt{\al})=(\al,-\langle\al,\mu\rangle)<0$, because $\mu\in
C_0$ is regular, thus $\langle\al,\mu\rangle>0$.
\end{Emp}


\begin{Lem} \label{L:border}
Assume that $w',w''\in C\cap \wt{W}$ for some $C\in\C{C}$ and
$w'<w''$. Then

\smallskip

(a) for every $u\in W$, we have $uw'<uw''$;

\smallskip

(b) there exists a sequence
$w'<w_1<\ldots<w_n=w''$ such that $w_i\in C$ and $l(w_i)=l(w')+i$
for each $i$;

\smallskip

(c) for every $\mu\in\La\cap C$ and $w\in \wt{W}\cap C$, we have
$l(\mu w)=l(\mu)+ l(w)$.


\end{Lem}

\begin{proof}
(a) By induction, it is enough to show that for every element $s\in S$ we
have $sw'<sw''$. By Section~\re{border}(b), it is enough to show that
$w'<sw'$ if and only if $w''<sw''$. Let $u\in W$ be such that
$C=u(C_0)$. Then it follows from Section~\re{border}(e) that each
condition $w'<sw'$ and $w''<sw''$ is equivalent to $u<su$.

\smallskip

(b) Using part~(a) and Section~\re{border}(e), we may assume that $C=C_0$. If the difference
$l(w'')-l(w')=1$, there is nothing to prove, so we can assume that
$l(w'')-l(w')>1$. By induction, it is enough to show the existence
of $w\in C_0$ such that $w'<w<w''$.

Choose $s\in\wt{S}$ such that $w''s<w''$. Then $w''s\in C_0$ by
Section~\re{border}(f). If $w'<w''s$, then  $w:=w''s$ does the job. If
not, then by Section~\re{border}(a) we get $w's<w'$ and $w's<w''s$. Then
by Section~\re{border}(f), we have  $w's\in C_0$, so by induction on $l(w'')$,
there exist $w\in C_0$ such that $w's<w<w''s$.

If $ws<w$, then it follows from Section~\re{border}(a) that $w'\leq
w<w''s$, contradicting our assumption. Hence we may assume that
$ws>w$, in which case by Section~\re{border}(a) we have $w'<ws<w''$, thus
it is enough to show that $ws\in C_0$.

Assume that $ws\notin C_0$. Since $w\in C_0$, this would imply
that there exists a simple root $\al$ of $C_0$ such that
$ws=s_{\al}w$. Then we have  $w'<s_{\al}w$ and $w'\in C_0$
therefore by Section~\re{border}(c) that $w'\leq w<w''s$, contradicting
the assumption.

\smallskip

(c)  Using Sections~\re{border}(d),(e), we can assume that $C=C_0$. Now
the proof goes by induction on $l(w)$. Choose $s\in\wt{S}$ such
that $ws<w$. Then $ws\in C_0$ by Section~\re{border}(f), hence by the
induction hypothesis, we have
\[
l(\mu ws)=l(\mu)+l(ws)=l(\mu)+l(w)-1.
\]
Thus it is enough to show that
$\mu ws<\mu w$.

Let $\al$ be a simple affine root such that $s=s_{\al}$. Then
$\wt{\beta}:=w(\al)<0$, because $ws<w$, and we want to show that
$\mu(\wt{\beta})=\mu w(\al)<0$. Write $\wt{\beta}$ in the form
$(\beta,n)$, where $\beta\in\Phi$. Then
$\mu(\wt{\beta})=\wt{\beta}-\langle\beta,\mu\rangle$, so it
remains to show that $\langle\beta,\mu\rangle\leq 0$.

Since $\wt{\beta}<0$, we get $n\leq 0$, therefore
$w^{-1}(\beta)=\al-n>0$. This implies that $\beta\in\Phi_{C_0}$,
because $w\in C_0$, hence $\langle\beta,\mu\rangle\leq 0$, because
$\mu\in C_0$.
\end{proof}





\subsection{Orderings on affine Weyl groups}

\begin{Not} \label{N:ord}
(a) Let $\wt{\al}\in\wt{\Phi}$ and $w\in\wt{W}$. We say that
$s_{\wt{\al}}w<_{\wt{\al}}w$ if $w^{-1}(\wt{\al})>0$.

\smallskip

(b) Let $\Phi'\subseteq \Phi$ be a subset, and
$w',w''\in\wt{W}$. We say that $w''<_{\Phi'} w'$ if there
exist affine roots $\wt{\al}_1,\ldots,\wt{\al}_n\in\wt{\Phi'}$ such that
$s_{\wt{\al}_i}\ldots
s_{\wt{\al}_1}w'<_{\wt{\al}_i}s_{\wt{\al}_{i-1}}\ldots
s_{\wt{\al}_1}w'$ for all $i$, and $w''=s_{\wt{\al}_n}\ldots s_{\wt{\al}_1}w'$.
For $\al\in\Phi$, we write $w''<_{\al} w'$ instead of $w''<_{\{\al\}} w'$.

\smallskip

(c) Let $\Phi'\subseteq \Phi$, and $x',x''\in V$. We say that
$x''<_{\Phi'} x'$ if the difference $x'-x''$ is a positive linear
combination of elements $\check{\al}$ with $\al\in\Phi'$.
For $\al\in\Phi$, we write $x''<_{\al} x'$ instead of $x''<_{\{\al\}} x'$.

\smallskip

 (d) For each $C\in\C{C}$, $\psi\in\Psi$ (and $\psi\in C$),
we write  $<_C$ (resp. $<_{\psi}$, resp. $<_{C^{\psi}}$)
instead of $<_{\Phi_C}$ (resp. $<_{\Phi(\psi)}$, resp.
$<_{\Phi^{\psi}(C^{\psi})}$).
\end{Not}

The following lemma explain a connection between these
notions.

\begin{Lem} \label{L:ord}
(a) For each $\wt{\al}=(\al,n)\in\wt{\Phi}$ and $w\in\wt{W}$, we have
$s_{\wt{\al}}w<_{\wt{\al}}w$ (see Section~\ref{N:ord}(a)) if and only if
$s_{\wt{\al}}w(x)<_{\al}w(x)$ (see Section~\ref{N:ord}(c)) for all $x\in A_0$.

\smallskip

(b) For each $\al\in\Phi$ and $w\in\wt{W}$, we have
$w<_{\al}\check{\al}w$ (see Section~\ref{N:ord}(b))
\end{Lem}

\begin{proof}
(a) Fix $x\in A_0$. Then $w^{-1}(\wt{\al})>0$ if and only if
$w^{-1}(\wt{\al})(x)=\wt{\al}(w(x))>0$. Thus $s_{\wt{\al}}w<_{\wt{\al}}w$ if and only if
$s_{\wt{\al}}w(x)=w(x)-\wt{\al}(w(x))\check{\al}<_{\al}w(x)$.

\smallskip

(b) Let $r\in\B{Z}$ such that the affine root $\wt{\al}=(\al,r)$ satisfies $0<\wt{\al}(\check{\al}w(x))<1$.
Since $\wt{\al}(s_{\wt{\al}}(\check{\al}w(x)))=- \wt{\al}(\check{\al}w(x))$, we get
$0<(\wt{\al}+1)(s_{\wt{\al}}(\check{\al}w(x)))<1$. Thus, by the observation of part~(a) we have
$w=s_{\wt{\al}+1}s_{\wt{\al}}(\check{\al}w)<_{\wt{\al}+1}s_{\wt{\al}}(\check{\al}w)<_{\wt{\al}}\check{\al}w$, hence $w<_{\al}\check{\al}w$.
\end{proof}

\begin{Cor} \label{C:ord}
(a) For each $w,w'\in\wt{W}$ and $\al\in\Phi$, we have
$w<_{\al}w'$ if and only if we have $w\in\wt{W}_{\al}w'$ and
$w(x)<_{\al}w'(x)$ for all $x\in A_0$.

\smallskip

(b) For each $\Phi'\subseteq\Phi$ and $w,w'\in\wt{W}$ with
$w<_{\Phi'} w'$, we have $w(x)<_{\Phi'} w'(x)$ for each $x\in A_0$,
hence  $\pi(w)\leq_{\Phi'} \pi(w')$ in the sense of Section~\ref{N:ord}(c).

\smallskip

(c) Let $\Phi'\subseteq\Phi$ has a property that if $\mu\in\La$ is a positive linear
combination of elements $\check{\al}$ with $\al\in\Phi'$, then $\mu$ is a finite sum of
elements $\check{\al}$ with $\al\in\Phi'$. Then for every $\mu,\mu'\in\La$, we have
$\mu<_{\Phi'}\mu'$ in the sense of Section~\ref{N:ord}(b) if and only if
$\mu<_{\Phi'}\mu'$ in the sense of Section~\ref{N:ord}(c).
\end{Cor}

\begin{proof}
(a) If  $w<_{\al}w'$, then $w\in\wt{W}_{\al}w'$ (by definition),
and $w(x)<_{\al}w'(x)$ for all $x\in A_0$ (by \rl{ord}(a)). Conversely, assume that $w=uw'$ with $u\in\wt{W}_{\al}$ such that
$w(x)<_{\al}w'(x)$ for all $x\in A_0$. Then we have either
$u=s_{\wt{\al}}$ or $u=\check{\al}^m$ for some $m\in\B{Z}_{<0}$. In the
first case, we have $w<_{\al}w'$ by \rl{ord}(a), while in the
second one we have $w<_{\al}w'$ by \rl{ord}(b).

\smallskip

(b) By definition, it is enough to assume that
$w=s_{\wt{\al}}w'<_{\wt{\al}}w'$. In this case, the first assertion follows
from \rl{ord}(a). Next since $0\in V$ lies in the closure of $A_0\subseteq V$,
the second one follows from the equality $\pi(w)=w(0)$.

\smallskip

(c) Assume that $\mu<_{\Phi'}\mu'$ in the sense of Section~\ref{N:ord}(c).
By our assumption of $\Phi'$, we may assume that $\mu=\mu'-\check{\al}$ for some
$\al\in\Phi'$. In this case, it follows from \rl{ord}(b) that $\mu<_{\Phi'}\mu'$ in the sense of Section~\ref{N:ord}(b).
The converse assertion follows from part~(b).
\end{proof}

\begin{Emp} \label{E:order}
{\bf Remarks.} (a) Let $\Phi'\subseteq\Phi$, let $w\in\wt{W}$, and let $\ov{w}\in W$ be the image of $w\in\wt{W}$
under the projection $\wt{W}\to W$. Then it follows from definition that for every $w'\leq_{\Phi'} w''$ we have $w w'\leq_{\ov{w}(\Phi')} ww''$.
In particular,

\smallskip

\quad\quad (i) for every $\mu\in\La$, we have $w'\leq_{\Phi'} w''$ if and only if
$\mu w'\leq_{\Phi'} \mu w''$;

\smallskip

\quad\quad (ii)  for every $u\in W$, we have $w'\leq_{\Phi'} w''$ if and only if
$uw'\leq_{u(\Phi')} uw''$.

\smallskip

(b) Note that for each $\al\in\Phi$, the subset $\Phi':=\{\al\}$ satisfies the assumption of \rco{ord}(c).
\end{Emp}

\begin{Prop} \label{P:order}
Let $w',w''\in\wt{W}$, and let $C$ be a Weyl chamber.

\smallskip

Then  $w'\leq_C w''$ if and only if for every sufficiently regular
$\mu\in \La\cap C$ we have $\mu w'\leq\mu w''$, that is, there
exists $\mu\in \La\cap C$ such that $\mu'\mu w'\leq\mu'\mu w''$
for every $\mu'\in\La\cap C$.
\end{Prop}
\begin{proof}

First we claim that for every $w',w''\in \wt{W}\cap C$ and
$\wt{\al}\in\wt{\Phi}$ such that $w'=s_{\wt{\al}}w''$ we have
$w'<_C w''$ if and only if $w'< w''$.

Replacing $\wt{\al}$ by $-\wt{\al}$, if necessary, we may assume that $\wt{\al}=\al+n$
with $\al\in\Phi_C$. Then $w'<_C w''$ holds  is and only if
$w''^{-1}(\wt{\al})>0$. On the other hand, since $w',w''\in C$, we get that
$(w'')^{-1}(\al)>0$ and $(w')^{-1}(\al)>0$. Since $s_{\wt{\al}}(\wt{\al})=-\wt{\al}$, we get
$s_{\wt{\al}}(\al)=-\al-2n$, therefore $(w')^{-1}(\al)
=(w'')^{-1}s_{\wt{\al}}(\al)=-(w'')^{-1}(\al)-2n>0$. This together
with $(w'')^{-1}(\al)>0$ implies that $n<0$, thus $\wt{\al}<0$.
Therefore $w'< w''$ holds if and only if $w''^{-1}(\wt{\al})>0$.

\smallskip

Now we are ready to show our assertion. Assume that $w'\leq_C w''$
and we are going to show that for each sufficiently regular
$\mu\in\La\cap C$ we have $\mu w'\leq\mu w''$. By induction, we
can assume that $w'=s_{\wt{\al}}w''$ for some
$\wt{\al}\in\wt{\Phi}$. Choose $\mu\in\La\cap C$ sufficiently
regular so that $\mu w',\mu w''\in C$. Then $\mu w'\leq _C \mu
w''$ (by Remark~\re{order}(a)(i)), and $\mu w'=s_{\mu(\wt{\al})}\mu w''$.
Hence, by the shown above,  $\mu w'\leq \mu w''$.

Conversely, assume that for every sufficiently regular element
$\mu\in\La\cap C$ we have $\mu w'\leq \mu w''$, and we want
to show that $w'\leq_C w''$. Replacing $w'$ and $w''$ by $\mu w'$
and $\mu w''$, respectively, and using Remark~\re{order}(a)(i), we may
assume that $w',w''\in C$ and $w'\leq w''$. Then using \rl{border}(b) we may assume in addition that $w'=s_{\wt{\al}}w''$. Then, by
the shown above, $w'\leq_C w''$.
\end{proof}

\begin{Cor} \label{C:order}
Let $w',w''\in\wt{W}$, and let $C$ be a Weyl chamber.

\smallskip

(a) If $w'\leq_C w''$ and $w'\in C$, then $w'\leq w''$.

\smallskip

(b) If $w'\leq w''$ and $w''\in C$, then $w'\leq_C w''$.

\smallskip

(c) If $w',w''\in C$, then $w'\leq_C w''$ if and only if $w'\leq
w''$.
\end{Cor}

\begin{proof}
(a)  By \rp{order}, there exists $\mu\in C$ such that $\mu w'\leq
\mu w''$. Since $\mu, w'\in C$, the assertion follows from
\rl{border}(c) and Section~\re{border}(c).

\smallskip

(b) Using \rl{border}(c) and Section~\re{border}(c), we conclude that
$\mu w'\leq \mu w''$ for every $\mu\in\La\cap C$. Therefore we get
$w'\leq_C w''$ by \rp{order}.

\smallskip

(c) follows from parts~(a) and (b).
\end{proof}

\begin{Lem} \label{L:order1}
Let $\psi\in\Psi$, $w',w''\in \wt{W}^{\psi}$, $C\owns\psi$ and
$w\in\wt{W}$.

\smallskip

(a) We have $w''w\leq_{C} w'w$ if and only if $w''w\leq_{C^{\psi}}
w'w$.

\smallskip

(b) If $w\in\wt{W}_{\psi}$, then  $w''w\leq_{C} w'w$ if and only
if $w''\leq_{C^{\psi}} w'$.
\end{Lem}

\begin{proof}
(a) Since $(\Phi^{\psi})_{C^{\psi}}\subseteq \Phi_C$, the ``if''
assertion is obvious. Conversely, assume that $w''w\leq_C w'w$.
Then there exists affine roots
\[
\wt{\beta}_1=(\beta_1,n_1),\ldots,\wt{\beta}_r=(\beta_r,n_r)\in\wt{\Phi}_C
\]
such that $w''w=s_{\wt{\beta}_n}\ldots s_{\wt{\beta}_1}w'w$, and
$s_{\wt{\beta}_i}\ldots s_{\wt{\beta}_1}w'w<_{\wt{\beta}_i}
s_{\wt{\beta}_{i-1}}\ldots s_{\wt{\beta}_1}w'w$ for all $i$. Then
for every $x\in A_0$, the difference $w'w(x)-w''w(x)$ is a
positive linear combination of the $\beta_i$'s (by \rl{ord}(a)).

Since $w''w\in \wt{W}^{\psi}w'w$, we conclude that
$w'w(x)-w''w(x)$ is a linear combination of roots of
$\Phi^{\psi}$. Therefore each $\beta_i$ is a root of
$\Phi^{\psi}$, thus $\beta_i\in (\Phi^{\psi})_{C^{\psi}}$. But
this implies that $w''w\leq_{C^{\psi}} w'w$.

\smallskip

(b) By part~(a), we have to show that  $w''w\leq_{C^{\psi}} w'w$ if and
only if $w''\leq_{C^{\psi}} w'$. Thus we can assume that
$w''=s_{\wt{\beta}}w'$ for some
$\wt{\beta}\in\wt{\Phi}^{\psi}$. In other words, we have to show
that $w'^{-1}(\wt{\beta})\in\wt{\Phi}^{\psi}_{>0}$ if and only if
$w^{-1}(w'^{-1}(\wt{\beta}))\in \wt{\Phi}_{>0}$. But this follows
from the assumption that $w\in\wt{W}_{\psi}$.
\end{proof}

\subsection{Admissible tuples}

\begin{Def} \label{D:adm}
(a) We say that a tuple $\ov{\mu}=\{\mu_C\}_{C\in\C{C}}\in
V^{\C{C}}$ {\em admissible} (resp. {\em quasi-admissible}), if for
every  $C\in \C{C}$ and
 $\al\in\Dt_{C}$, the difference $\mu_C-\mu_{s_{\al}(C)}$ belongs to
$\B{R}_{\geq 0}\check{\al}$ (resp. $\B{R}\check{\al}$).

\smallskip

(b) A tuple $\ov{w}=\{w_C\}_C\in\wt{W}^{\C{C}}$  is called {\em
admissible} (resp. {\em quasi-admissible}), if for every
$C\in\C{C}$ and $\al\in\Dt_{C}$, we have
$w_{s_{\al}(C)}\leq_{\al}w_C$ (resp. $w_{s_{\al}(C)}\in
\wt{W}_{\al}w_C$).
\end{Def}

\begin{Emp} \label{E:qadm}
{\bf Remarks.} (a) It follows from \rco{ord}(b) that if
$\ov{w}\in\wt{W}^{\C{C}}$  is (quasi)-admissible, then tuple
$\pi(\ov{w})\in\La^{\C{C}}\subseteq V^{\C{C}}$ is
(quasi)-admissible as well. Moreover, it follows from \rco{ord}(c) and Section~\re{order}(b) that
a tuple $\ov{\mu}\in\La^{\C{C}}$ is (quasi)-admissible as an
element of $\wt{W}^{\C{C}}$ if and only if it is such as an
element of $V^{\C{C}}$.

\smallskip

(b) The set of quasi-admissible tuples in $V$ (resp. $\La$) can be
naturally identified with $\B{R}^{\Psi}$ (resp. $\B{Z}^{\Psi}$).
Indeed, for each quasi-admissible $\ov{\mu}\in V^{\C{C}}$ and
every $\psi\in\Psi$, the element
$\mu_{\psi}:=\langle\psi,\mu_C\rangle$ does not depend on
$C\owns\psi$. To see this, we observe that for every pair of Weyl chambers $C,C'\owns\psi$ there exists
$w\in W_{\Phi^{\psi}}$ such that $C'=w(C)$. Therefore $\ov{\mu}$ defines a tuple
$\{\mu_{\psi}\}_{\psi\in\Psi}\in\B{R}^{\Psi}$.

\smallskip

Conversely, every tuple $\{\mu_{\psi}\}\in \B{R}^{\Psi}$ gives
rise to a quasi-admissible tuple $\ov{\mu}\in V^{\C{C}}$ defined
by the rule $\mu_C:=\sum_{\al_i\in\Dt_C}\mu_{\psi_i}\check{\al}_i$, where
$\psi_i\in\Psi_C$ is the fundamental weight corresponding to
$\al_i\in\Dt_C$.

\smallskip

(c) From now on we will not distinguish between a quasi-admissible
tuple $\{\mu_C\}_C$ in $V^{\C{C}}$ (resp. $\La^{\C{C}}$) and the
corresponding tuple $\{\mu_{\psi}\}_{\psi}$ in $\B{R}^{\Psi}$
(resp. $\B{Z}^{\Psi}$). In particular, for every $\psi\in\Psi$
we denote by $\ov{e}_{\psi}\in \La^{\C{C}}$ the quasi-admissible
tuple, corresponding to the standard vector $\ov{e}_{\psi}\in \B{Z}^{\Psi}$,
given by the rule $\ov{e}_{\psi}(\psi')=\dt_{\psi,\psi'}$.

\smallskip

(d) Arguing as in part~(b), for each quasi-admissible $\ov{w}\in\wt{W}^{\C{C}}$
and every $\psi\in\Psi$, the class $[w_C]\in\wt{W}^{\psi}\bs\wt{W}$, hence also element
$(w_C)_{\psi}\in\wt{W}_{\psi}$ (see Section~\re{fund}(d)) does not depend
on $C\owns\psi$. We will denote this element by $w_{\psi}=\ov{w}_{\psi}$.

\smallskip

(e) The set of quasi-admissible (resp. admissible) tuples in $V^{\C{C}}$
(resp. $\La^{\C{C}}$) is an ordered group (resp. an ordered  monoid) with
respect to the coordinatewise addition in $V$ (resp. $\La$). Moreover,
the identification of part~(b) identifies this group with $\B{R}^{\Psi}$
(resp. $\B{Z}^{\Psi}$).

\smallskip

(f) Using Remark~\re{order}(a)(i) and \rl{ord}(b), for every (quasi)-admissible tuples
$\ov{\mu}\in\La^{\C{C}}$ and $\ov{w}\in\wt{W}^{\C{C}}$, the tuple
$\ov{\mu}\cdot \ov{w}:=\{\mu_Cw_C\}_C\in\wt{W}^{\C{C}}$ is
(quasi)-admissible as well. In particular, for every $\mu\in\La$
and (quasi)-admissible tuple $\ov{w}\in\wt{W}^{\C{C}}$, the tuple
$\mu\ov{w}:=\{\mu w_C\}_C$ is (quasi)-admissible.
\end{Emp}

\begin{Emp} \label{E:exam}
{\bf Examples.} (a) Every $\mu\in C_0\subseteq V$ gives rise to an
admissible tuple $\ov{\mu}\in V^{\C{C}}$ defined by the rule
$\mu_{u(C_0)}:=u(\mu)$ for all $u\in W$.

\smallskip

(b) Consider the tuple $\ov{w}_{\on{st}}\in
W^{\C{C}}\subseteq\wt{W}^{\C{C}}$, defined by the rule
$(w_{\on{st}})_{u(C_0)}=u$. Then $\ov{w}_{\on{st}}$ is admissible.
Indeed, by definition we have to show that for every $u\in W$ and $\al\in \Phi_{u(C_0)}$
we have $s_{\al}u<_{\al}u$, that is, $u^{-1}(\al)>0$. Since $u^{-1}(\al)\in\Phi_{C_0}$, we are done.
\end{Emp}

The following characterization of admissible tuples will be
crucial for the rest of the paper.

\begin{Lem} \label{L:adm}
A tuple $\ov{w}\in\wt{W}^{\C{C}}$ (resp. $\ov{\mu}\in V^{\C{C}}$)
is admissible if and only if for all $C,C'\in\C{C}$, we have
$w_{C}\leq_{C'} w_{C'}$ (resp. $\mu_{C}\leq_{C'}\mu_{C'}$).
\end{Lem}

\begin{proof}
We will only prove the assertion for $\ov{w}$, while the other
case is similar, but easier. Assume first that $\ov{w}$ is
admissible, and we want to show that for every two Weyl chambers
$C$ and $C'$ we have $w_{C}\leq_{C'} w_{C'}$. Using Remark~\re{order}(a)(ii), we may assume that $C'=C_0$. Let $u\in W$ be such that
$C=u(C_0)$, choose a reduced decomposition $u=s_1\ldots s_n$ of
$u$, and for each $j=1,\ldots,n$ we set $u_j:=s_1\ldots s_j$ and
$C_j:=u_j(C_0)$. It is enough to show that $w_{C_{j+1}}\leq_{C_0}
w_{C_j}$ for each $j$.

Let $\al_{j+1}\in\Dt_{C_0}$ be such that $s_{j+1}=s_{\al_{j+1}}$.
By construction, we obtain that $u_{j+1}=u_j s_{i+1}>u_j$, hence
$u_j(\al_{j+1})\in\Phi_{C_0}$. Also since $\al_{j+1}\in\Dt_{C_0}$, we get
that $u_j(\al_{j+1})\in\Dt_{C_j}$.  Since
$C_{j+1}=s_{u_j(\al_{j+1})}(C_j)$, the admissibility assumption implies that
$w_{C_{j+1}}\leq_{u_j(\al_{j+1})}w_{C_j}$, thus
$w_{C_{j+1}}\leq_{C_0} w_{C_j}$, because
 $u_j(\al_{j+1})\in\Phi_{C_0}$.

\smallskip

Conversely, assume that $w_{C}\leq_{C'} w_{C'}$ for all
$C,C'\in\C{C}$. Choose $C\in\C{C}$, $\al\in\Dt_C$, and set
$C'=s_{\al}(C)$. Since $w_{C'}\leq_C w_C$, there exist a tuple
of affine roots
$\wt{\beta}_1=(\beta_i,n_i),\ldots,\wt{\beta}_r=(\beta_r,n_r)\in
\Phi_{C}$ such that $w_{C'}=s_{\wt{\beta}_r}\ldots
s_{\wt{\beta}_1} w_C$, and $s_{\wt{\beta}_i}\ldots
s_{\wt{\beta}_1}w_C<_{\wt{\beta}_i} s_{\wt{\beta}_{i-1}}\ldots
s_{\wt{\beta}_1}w_C$ for all $i$.

Therefore for each $x\in A_0$, the difference $w_C(x)-w_{C'}(x)$ is
a positive linear combination of the $\beta_i$'s (by \rl{ord}(a)), hence a non-positive linear combination $C$-positive roots.
On the other hand, since $w_C\leq_{C'}w_{C'}$, the difference
$w_C(x)-w_{C'}(x)$ is also a non-negative linear combination
$C'$-negative roots.

Combining these two statements, we conclude that
$w_C(x)-w_{C'}(x)$ has to be a positive multiple of $\al$. Hence
all the $\beta_i$'s have to be $\al$, thus $w_{C'}\leq_{\al} w_C$.
\end{proof}

\begin{Not} \label{N:kreg}
(a) Let $m\in\B{R}$ and $C\in\C{C}$. We say that $\mu\in V$ is
{\em $(C,m)$-regular}, if  $\langle \al,\mu\rangle\geq m$
 for every $\al\in\Phi_{C}$. We say that $w\in\wt{W}$ is
{\em $(C,m)$-regular}, if $\pi(w)=w(0)\in\La$ is  $(C,m)$-regular.

\smallskip

(b) Let $m\in\B{R}$. We say that a tuple $\ov{\mu}\in V^{\C{C}}$
is {\em $m$-regular}, if $\mu_C$ is $(C,m)$-regular for every
$C\in\C{C}$. We say that a tuple $\ov{\mu}\in V^{\C{C}}$ is {\em
regular}, if it is $m$-regular for some $m>0$. A tuple
$\ov{w}\in\wt{W}^{\C{C}}$ is called {\em $m$-regular} (resp. {\em
regular}), if $\pi(\ov{w})\in \La^{\C{C}}\subseteq V^{\C{C}}$ is
$m$-regular (resp. regular).

\smallskip

(c)  For $\ov{\mu},\ov{\mu}'\in V^{\C{C}}$ (resp.
$\ov{w},\ov{w}'\in \wt{W}^{\C{C}}$), we will say that
$\ov{\mu}\leq\ov{\mu}'$ (resp. $\ov{w}\leq\ov{w}'$) if
$\mu_C\leq_C\mu'_C$ (resp. $w_C\leq_C w'_C$) for all $C\in\C{C}$.
Notice that $\ov{\mu}\leq\ov{\mu}'$ if and only if
$\mu_{\psi}\leq\mu'_{\psi}$ for all $\psi\in\Psi$.

\smallskip

(d) For $\ov{\mu}\in V^{\C{C}}$, we define by $V^{\leq\ov{\mu}}$
the set of all $x\in V$ such that $x\leq_C\mu_C$ for all
$C\in\C{C}$.

\smallskip

(e) For every $\ov{w}\in\wt{W}^{\C{C}}$ and every $\psi\in\Psi$ we
define $\ov{w}^{\psi}\in (\wt{W}^{\psi})^{\C{C}^{\psi}}$ by the
rule $(w^{\psi})_{C^{\psi}}=(w_C)^{\psi}$ for each $C\owns\psi$
(see Section~\re{fund}).
\end{Not}

\begin{Lem} \label{L:ord1}
(a) If $\ov{w}\in\wt{W}^{\C{C}}$ (resp. $\ov{\mu}\in V^{\C{C}}$)
is quasi-admissible and regular, then it is admissible.

\smallskip

(b) If $\ov{\mu}\in V^{\C{C}}$ is quasi-admissible and regular, then
for every $\mu\in\Psi$ we have $\mu_{\psi}>0$ (see Section~\re{qadm}(b)).

\smallskip

(c) If the tuple $\ov{w}\in\wt{W}^{\C{C}}$ is admissible, then the
tuple  $\ov{w}^{\psi}$ is admissible as well.

\smallskip

(d) If $\ov{w}\in\wt{W}^{\C{C}}$ is $(m+1)$-regular, then
$\ov{w}^{\psi}$ is $m$-regular.
\end{Lem}

\begin{proof}

(a) We will only show the assertion for $\ov{w}$. Fix $C\in\C{C}$,
let $\al\in\Phi_C$, and set $C'=s_{\al}(C)$. We want to show that
$w_{C'}\leq_{\al} w_C$. Since $\ov{w}$ is
quasi-admissible, we get $w_{C'}\in\wt{W}_{\al}w_{C}$. Therefore
for every $x\in A_0$ we have $w_{C'}(x)=w_C(x)-a\check{\al}$ for
some $a\in\B{R}$. Since $\ov{w}$ is regular we conclude $\langle
\al,w_C(x)\rangle>0$ and $\langle \al,w_{C'}(x)\rangle<0$. Thus
$a>0$, hence $w_{C'}\leq_{\al} w_C$ (by \rco{ord}(a)).

\smallskip

(b) Since $\mu_{\psi}=\lan\psi,\mu_C\ran$ for every Weyl chamber $C\owns \psi$ (by definition), we have
$\lan\al,\mu_C\ran>0$ for every $\al\in\Dt_C$  (since $\ov{\mu}$ is regular), and $\psi=\sum_{\al\in\Dt_C} c_{\al}\al$
with $c_{\al}\geq 0$ for all $\al\in\Dt_C$, the assertion follows.

\smallskip

(c) follows from \rl{order1}(b) and \rl{adm}.

\smallskip

(d) Let $C\owns \psi$ be a Weyl chamber. We have to show that if
$w$ is $(C,m+1)$-regular, and $\psi\in C$, then $w^{\psi}$ is
$(C^{\psi},m)$-regular.

Notice that if $w$ is $(C,m+1)$-regular, then for all
$\al\in\Phi_C$ and $x\in A_0$ we have $\langle
w^{-1}(\al),x\rangle=\langle\al,w(x)\rangle>m$, or equivalently,
$w^{-1}(\al)-m\in\wt{\Phi}_{>0}$. Conversely, if
$\langle\al,w(x)\rangle>m$ for all $\al\in\Phi_C$, then $\langle\al,w(0)\rangle\geq m$ for all $\al\in\Phi_C$, thus
$w$ is $(C,m)$-regular.

Thus it suffices to show that for every
$\al\in(\Phi^{\psi})_{C^{\psi}}$, we have
$w^{-1}(\al)-m\in\wt{\Phi}_{>0}$ if and only if
$(w^{\psi})^{-1}(\al)-m\in\wt{\Phi}^{\psi}_{>0}$. Since
$w=w^{\psi}w_{\psi}$, the assertion follows from the fact that
$w_{\psi}\in\wt{W}_{\psi}$.
\end{proof}

The following lemma will be used in \rl{trunc}.

\begin{Lem} \label{L:k-reg}
Let $\ov{\mu}\in V^{\C{C}}$ be regular, and $\psi\in\Psi$. Then for every $x\in V^{\leq\ov{\mu}}$ and $\al\in\Phi$ such that
$\langle\al,\check{\psi}\rangle>0$ and $\langle \psi,x\rangle=\mu_{\psi}$,
we have $\langle\al,x\rangle>0$.
\end{Lem}

\begin{proof}
By \rl{ord1}(a), the tuple $\ov{\mu}$ is admissible. Thus it is
well-known (see, for example \cite{Be} or argue as in
\rl{adm}), that the intersection of $V^{\leq\ov{\mu}}$ with the
set of $x\in V$ such that $\langle \psi,x\rangle=\mu_{\psi}$ is
equal to the convex hull of $\{\mu_C\}_{C\owns\psi}$.

Therefore it is enough to show that for every
 Weyl chamber $C\owns\psi$, we have $\langle\al,\mu_C\rangle>0$. Since
tuple $\ov{\mu}$ is regular, it is enough to show that $\langle\al,y\rangle>0$ for some $y\in C\subseteq V^*$.
But this follows from  our assumption $\langle\al,\check{\psi}\rangle>0$ together with observation
that $\check{\psi}\in\ov{C}$ (see Section~\re{fund}(b)).
\end{proof}

The following very important technical result will be used in \rp{ind}.

\begin{Lem} \label{L:order2}
 Let $\ov{u}\in\wt{W}^{\C{C}}$ be admissible, and $\psi\in\Psi$.
Then there exists $m\in\B{N}$ such that for every $m$-regular
admissible tuple $\ov{w}\in\wt{W}^{\C{C}}$ and every $\mu\in\La$
such that $\mu u_C\leq_{C^{\psi}} w_C$ for each $C\owns \psi$, we
have $\mu u_C\leq_C w_C$ for each $C$.
\end{Lem}

\begin{proof}
First we claim that there exists an admissible tuple
$\ov{\mu}\in\La^{\C{C}}$ such that $\mu_C^{-1}u_C\leq_C
\mu^{-1}_{C'}u_{C'}$ for every $C,C'\in \C{C}$. Indeed, $\ov{u}$
is admissible, hence for each $C\in\C{C}$, $\al\in\Dt_C$ and $x\in
A_0$ we have
$u_C(x)-u_{s_{\al}(C)}(x)=m_{C,\al,x}\check{\al}$ for some constant $m_{C,\al,x}\geq 0$
(use \rl{ord}(a)). Let $m'$ be the supremum of the $m_{C,\al,x}$'s, choose
$\mu\in C_0\cap\La$ such that $\langle\al,\mu\rangle\geq m'$ for
all $\al\in\Dt_{C_0}$, and let $\ov{\mu}\in \La^{\C{C}}$ be the
tuple, corresponding to $\mu$ as in Section~\re{exam}(b).

We claim that $\mu_C^{-1}u_C\leq_C \mu^{-1}_{C'}u_{C'}$ for every
$C,C'\in \C{C}$. Indeed, arguing as in \rl{adm} word-by-word, it
is enough to check that $\mu_C^{-1}u_C\leq_{\al}
\mu^{-1}_{C'}u_{C'}$ for all $C\in\C{C},\al\in\Dt_C$ and
$C'=s_{\al}(C)$. Then by \rco{ord}(a), it is enough to check that
$\mu_C^{-1}u_C(x)\leq_{\al} \mu^{-1}_{C'}u_{C'}(x)$ for each $x\in
A_0$. By construction, we have
$\mu^{-1}_{C'}u_{C'}(x)-\mu_C^{-1}u_C(x)=
(\langle\al,\mu_C\rangle-m_{C,\al,x})\check{\al}$,
so the assertion follows from the fact that $m_{C,\al,x}<m'\leq
\langle\al,\mu_C\rangle$.

\smallskip

Denote $m$ be the maximum of the $\langle\al,\mu_C\rangle+1$'s, taken
over all $C\in\C{C}$ and  $\al\in\Dt_{C}$. We claim that such an $m$
satisfies the required property.

To see this, we choose any $m$-regular admissible tuple $\ov{w}$, and we
claim that tuple $\ov{\mu}^{-1}\cdot\ov{w}=\{\mu_C^{-1}w_C\}_C$ is
admissible. By Section~\re{qadm}(f), it is quasi-admissible, so by
\rl{ord1}(a), it is enough to show that it is regular. For every
$C\in\C{C}$, $\al\in\Dt_C$, we have
$\langle\al,\mu_C^{-1}\pi(w_C)\rangle=
\langle\al,\pi(w_C)\rangle-\langle\al,\mu_C\rangle>0$, because
$\langle\al,\pi(w_C)\rangle\geq m$ by $m$-regularity of $\ov{w}$,
and $\langle\al,\mu_C\rangle\leq m-1$, by construction.

Let $\mu\in\La$ be such that  $\mu
u_C\leq_{C^{\psi}} w_C$  for each $C\owns \psi$, and let $C'\in\C{C}$ be
arbitrary. We want to show that $\mu u_{C'}\leq_{C'} w_{C'}$. Using
Remark~\re{order}(a)(ii), it is enough to do it in the case $C'=C_0$, so using
Remark~\re{order}(a)(i), we have to show that
$\mu_{C_0}^{-1}\mu u_{C_0}\leq_{C_0}\mu_{C_0}^{-1}w_{C_0}$.

Choose $u\in W$ of minimal length such that $\psi_0:=u^{-1}(\psi)$
belongs to $\Psi_{C_0}$, and set $C:=u(C_0)$. Since $\ov{\mu}^{-1}\cdot\ov{w}$
is admissible, we conclude from \rl{adm} that
$\mu_{C}^{-1}w_{C}\leq_{C_0}\mu_{C_0}^{-1}w_{C_0}$, while by our
construction we get $\mu_{C_0}^{-1}u_{C_0}\leq_{C_0}
\mu_{C}^{-1}u_{C}$, hence $\mu_{C_0}^{-1}\mu u_{C_0}\leq_{C_0}
\mu_{C}^{-1}\mu u_{C}$ (by Remark~\re{order}(a)(i)).
 Thus it is enough to show that $\mu _C^{-1}\mu
u_C\leq_{C_0} \mu_C^{-1}w_{C}$, or, equivalently, that $\mu
u_C\leq_{C_0} w_{C}$.

Since $\psi_0\in C_0$, we get that $\psi\in C$. Hence, by our
assumption, $\mu u_C\leq_{C^{\psi}} w_{C}$. Therefore
to show that $\mu u_C\leq_{C_0} w_{C}$ it suffices to check that
$(\Phi^{\psi})_{C^{\psi}}\subseteq \Phi_{C_0}$.

If $\beta\in (\Phi^{\psi})_{C^{\psi}}$, then $u^{-1}(\beta)\in
(\Phi^{\psi_0})_{C_0^{\psi_0}}$. Since $u\in W$ is an element of minimal
length such that $\psi=u(\psi_0)$, we get that
$u((\Phi^{\psi_0})_{C_0^{\psi_0}})\subseteq\Phi_{C_0}$. In particular, we have
$\beta=u(u^{-1}(\beta))\in\Phi_{C_0}$.
\end{proof}

\section{Semi-infinite orbits in affine flag varieties}

\subsection{Definitions and basic properties}

\begin{Not} \label{E:groups}
(a) Let $k$ be an algebraically closed field,
$K:=k((t))$ the field of Laurent power series over $k$, and
$\C{O}=\C{O}_K=k[[t]]$ the ring of integers of $K$.
For every affine scheme $X$ over $\C{O}$ (resp. $K$) we denote
by $L^+X$ (resp. $LX$) the corresponding arc- (resp. loop-) space.

\smallskip

(b) Let $G$ be a semi-simple and simply connected group over $k$.
Fix a maximal torus $T\subseteq G$, let $\Phi=\Phi(G,T)$ be the
root system of $(G,T)$, let $W=W_{G}$ be the Weyl group of $G$,
and $\wt{W}=N_{LG}(LT)/L^+T$ the affine Weyl group of $G$.
Then in the notation of Section~\re{root} we have natural isomorphisms $\La\isom X_*(T)$
and $W_{\Phi}\isom W$. Moreover, the map
$\mu\mapsto\mu(t)$ defines an embedding $\La\hra LT$, which in
turn induces isomorphisms of groups $\La\isom LT/L^+T$ and
$\wt{W}_{\Phi}\isom\wt{W}$.
\end{Not}

\begin{Not} \label{E:notation}
(a) For every $C\in\C{C}$, let $B_C\subseteq G$ be the
Borel subgroup containing $T$ such that $\Phi(B_C,T)=\Phi_C$, and let
$U_C\subseteq B_C$ be the unipotent radical.

\smallskip

(b) Choose $C_0\in\C{C}_{\Phi}$ as in Section~\re{weyl}, let $T\subseteq B_0=B_{C_0}\subseteq G$
be the corresponding Borel subgroup, let $B_0^{-}\supseteq T$ be the opposite Borel subgroup, and let
$I\subseteq L^+G$ be the Iwahori subgroup, defined as the preimage of $B_0^{-}\subseteq G$ under
the projection $L^+G\to G$.

\smallskip

(c) For every $\al\in\Phi$ we have a natural isomorphism
$\exp_{\al}:\Lie U_{\al}\isom U_{\al}$. For
$\wt{\al}=(\al,n)\in\wt{\Phi}$, we set
$U_{\wt{\al}}:=\exp_{\al}(t^n \Lie U_{\al})\subseteq
L(U_{-\al})$, and $\wt{\al}':=(-\al,n)$.

\smallskip

(d) In the conventions of parts~(b),(c), we get the equality
$I=L^+T\cdot\prod_{\wt{\al}\in\wt{\Phi}_{>0}} U_{\wt{\al}'}$.
\end{Not}

\begin{Emp}
{\bf Affine flag varieties.}

\smallskip

 (a) Denote by $\Fl=\Fl_G$ the affine
flag variety $LG/I$ of $G$ over $k$, and by $\Gr=\Gr_G$ the affine
Grassmannian $LG/L^+G$.
We have a natural projection $\pr:\Fl\to
\Gr$. Note that both $\Fl$ and $\Gr$ are equipped with an action of the ind-group scheme $LG$, and that projection $\pr$ is $LG$-equivariant.

\smallskip

(b) The embedding $N_{LG}(LT)\hra LG$ induces embeddings $\wt{W}\hra \Fl$ and $\La\hra\Gr$  and we identify $\wt{W}$ (resp. $\La$) with its image in $\Fl$ (resp. $\Gr$). Furthermore, both $\Fl$  and $\Gr$ are equipped with the action of $T\subseteq L^+(T)\subseteq LG$,
and these identifications identify $\wt{W}$ (resp. $\La$) with the locus of $T$-fixed points $\Fl^T$ (resp. $\Gr^T$).

\smallskip

(c) Note that $\Fl$ decompose as a union $\Fl=\bigcup_{w\in\wt{W}}Iw$ of $I$-orbits, and for every $w\in\wt{W}$ we denote by $\Fl^{\leq w}\subseteq \Fl$ the closure of
the $I$-orbit $Iw\subseteq \Fl$. Then $\Fl^{\leq w}$ is a reduced projective
subscheme of $\Fl$ called {\em the affine Schubert variety}.

\smallskip

(d) Fix any $C\in\C{C}$. Then we have decompositions
$\Fl=\bigcup_{w\in\wt{W}}L(U_C)w$ and
$\Gr=\bigcup_{\mu\in\La}L(U_C)\mu$ by $L(U_C)$-orbits. For every $w\in\wt{W}$ (resp.
$\mu\in\La$) we denote by $\Fl^{\leq_{C}w}\subseteq \Fl$ (resp.
$\Gr^{\leq_{C}\mu}\subseteq \Gr$) the closure of the $L(U_C)$-orbit
$L(U_C)w\subseteq \Fl$ (resp. $L(U_C)\mu\subseteq \Gr$). We also set
$\Fl^{\leq'_{C}\mu}:=\pr^{-1}( \Gr^{\leq_{C}\mu})\subseteq \Fl$.
Notice that $\Fl^{\leq_{C}w}$, $\Gr^{\leq_{C}\mu}$ and
$\Fl^{\leq'_{C}w}$ are  closed reduced ind-subschemes.

\smallskip

(e) For every tuple $\ov{w}\in \wt{W}^{\C{C}}$ (resp. $\ov{\mu}\in
\La^{\C{C}}$), we denote by $\Fl^{\leq \ov{w}}$ (resp. $\Gr^{\leq
\ov{\mu}}$) the reduced intersection $\bigcap_C \Fl^{\leq_C w_C}$ (resp.
$\bigcap_C \Gr^{\leq_C \mu_C}$), and set $\Fl^{\leq'
\ov{\mu}}:=\pr^{-1}(\Gr^{\leq \ov{\mu}})$.
\end{Emp}

The following simple lemma will be used later.

\begin{Lem} \label{L:invar}
Let $Z\subseteq \Fl$ be a closed reduced $T$-invariant ind-subscheme, $C$
a Weyl chamber, and $w\in\wt{W}$. Then $Z\cap
L(U_C)w\neq\emptyset$ if and only if $w\in Z$.
\end{Lem}

\begin{proof}
Clearly if $w\in Z$, then $Z\cap L(U_C)w\neq\emptyset$.
Conversely, let $z$ be an element of $Z\cap L(U_C)w$, and pick
$u\in L(U_C)$ such that $z=uw$. For any $\mu\in\La=\Hom(\B{G}_m,T)$ and $a\in\B{G}_m$ we have  $\mu(a)(z)=(\mu(a)u\mu(a)^{-1})(w)$, because
$w\in\Fl$ is $T$-invariant, hence $(\mu(a)u\mu(a)^{-1})(w)\in Z$, because $Z$ is $T$-invariant.
Next, for $\mu\in\La\cap C$ the morphism $a\mapsto (\mu(a)u\mu(a)^{-1})(w):\B{G}_m\to Z\subseteq\Fl$ extends to the morphism
$\B{A}^1\to \Fl$, which sends $0$ to $w$. Since $Z\subseteq \Fl$ is closed, we conclude that $w\in Z$.
\end{proof}

\begin{Lem} \label{L:ineq}
Let $w,w'\in \wt{W}$, and let $C$ be a Weyl chamber.

\smallskip

(a) We have $w'\in \Fl^{\leq w}$ if and only if $w'\leq w$.

\smallskip

(b) If $w\in C$, then $I w\subseteq \Fl$ is contained in
$L(U_C)w\subseteq \Fl$.

\smallskip

(c) If $Iw\cap L(U_C)w'\neq\emptyset$, then $w'\leq w$.
\end{Lem}

\begin{proof}
(a) is a standard.

\smallskip

(b) Note that in the notation of Section~\re{notation}(c), for every $\wt{\al}\in\wt{\Phi}$ and $w\in\wt{W}$, we have $wU_{\wt{\al}}w^{-1}=U_{w(\wt{\al})}$  and $w(\wt{\al}')=w(\wt{\al})'$. Combining this with Section~\re{notation}(d), we see that  for every $w\in \wt{W}$ we
have
\begin{equation} \label{Eq:orbit}
Iw=\left(\prod_{\wt{\al}>0,w^{-1}(\wt{\al})<0}U_{\wt{\al}'}\right)w.
\end{equation}

Using formula \form{orbit}, it remains to check that every
$\wt{\al}=(\al,n)>0$ such that $w^{-1}(\wt{\al})<0$ satisfies
$U_{\wt{\al}'}\subseteq L(U_C)$, that is,
$-\al\in\Phi_C$. However, $n\geq 0$, because $\wt{\al}>0$. Therefore
$w^{-1}(\al)=w^{-1}(\wt{\al})-n<0$. Thus $w^{-1}(-\al)>0$, hence
 $-\al\in\Phi_C$ because $w\in C$.

\smallskip

(c) If $Iw\cap L(U_C)w'\neq\emptyset$, then  $\Fl^{\leq w}\cap
L(U_C)w'\neq\emptyset$. Since $\Fl^{\leq w}\subseteq \Fl$ is closed and
$T$-invariant, we get $w'\in \Fl^{\leq w}$ (by \rl{invar}), thus $w'\leq w$ (by part~(a)).
\end{proof}

The following proposition gives a geometric interpretation of the
ordering $\leq_C$, generalizing the well-known result (see, for example, \cite[Proposition~3.1]{MV}) for the affine Grassmannian.

\begin{Prop} \label{P:semiinf}
 For each $w',w''\in\wt{W}$ and every Weyl chamber $C\in\C{C}$,
we have $w'\leq_C w''$ if and only if  $w'\in \Fl^{\leq_C w''}$.
\end{Prop}

\begin{proof}
Assume that  $w'\leq_C w''$. Then by \rp{order} there exists
$\mu\in\La\cap C$ such that $\mu w'\leq\mu w''$ and $\mu w''\in
C$. Then $\mu w'\in \Fl$ lies in the closure of $I\mu w''\subseteq
\Fl$ (by \rl{ineq}(a)),  thus in the closure of $L(U_C)\mu
w''\subseteq \Fl$ (by \rl{ineq}(b)). Since $U_C$ is normalized by
$T$, this implies that $w'\in \Fl$ lies in the closure of
$\mu^{-1} L(U_C)\mu w'' =L(U_C)w''\subseteq \Fl$, that is, $w'\in \Fl^{\leq_C w''}$.

\smallskip

Conversely, assume that $w'\in \Fl$ lies in the closure of $L(U_C)
w''\subseteq \Fl$. Then there exists a closed subgroup scheme
$U'\subseteq L(U_C)$ such that $w'\in \Fl$ lies in the closure of $U'
w''\subseteq \Fl$. Then $\mu w'\in \Fl$ lies in the closure of
$\mu U' w''=(\mu U'\mu^{-1})\mu w''$ for every $\mu\in\La$.
However, if $\mu\in\La\cap C$ is sufficiently regular, then $\mu
U'\mu^{-1}\subseteq I$, thus $\mu w'\in \Fl$ lies in the closure of
$I\mu w''\subseteq \Fl$. This implies that $\mu w'\leq\mu w''$ (by \rl{ineq}(a)),
thus $w'\leq_C w''$ (by  \rp{order}).
\end{proof}

\begin{Cor} \label{C:adm}
(a) A tuple $\ov{w}\in\wt{W}^{\C{C}}$ is admissible if and only if
for every $C\in\C{C}$ the intersection $L(U_C)w_C\cap
\Fl^{\leq\ov{w}}$ is non-empty.

\smallskip

(b) For a tuple $\ov{u}$ and an admissible tuple $\ov{w}$,
we have $\Fl^{\leq\ov{w}}\subseteq \Fl^{\leq\ov{u}}$ if and only if
$\ov{w}\leq\ov{u}$.

\smallskip

(c) For a tuple $\ov{w}$, we have an inclusion
$\Fl^{\leq\ov{w}}\subseteq\bigcup_C \Fl^{\leq w_C}$. In particular, each
$\Fl^{\leq\ov{w}}\subseteq \Fl$ is a closed subscheme of finite type.

\smallskip

(d) Let $Z\subseteq \Fl$ be a closed $T$-invariant ind-subscheme. For every $z\in Z$, consider tuple
$\ov{u}=\ov{u}(z)\in\wt{W}^{\C{C}}$ defined by the rule that $z\in L(U_C)u_C$ for all $z\in\C{C}$. Then the tuple $\ov{u}$ is admissible, and $u_C\in \wt{W}\cap Z$ for all $C\in\C{C}$.

\smallskip

(e) In the situation of part~(d), we have an inclusion $Z\subseteq
\bigcap_C(\bigcup_{w\in\wt{W}\cap Z}\Fl^{\leq_C w})$.

\smallskip

(f) For every tuple $\ov{\mu}\in\La^{\C{C}}$, we have an equality
$\Fl^{\leq'\ov{\mu}}=\Fl^{\leq\ov{\mu}\cdot\ov{w}_{\on{st}}}$ (compare
Sections~\re{exam}(b) and \re{qadm}(f)).
\end{Cor}

\begin{proof}
(a) By \rl{adm} and \rp{semiinf}, a tuple $\ov{w}\in\wt{W}^{\C{C}}$ is admissible if and
only if $w_{C}\in \Fl^{\leq\ov{w}}=\bigcap_{C'\in\C{C}}\Fl^{\leq_{C'} w_{C'}}$ for each $C\in\C{C}$. Since
$\Fl^{\leq\ov{w}}\subseteq \Fl$ is closed and $T$-invariant, the
assertion now follows from \rl{invar}.

\smallskip

(b) The ``if'' assertion follows from \rp{semiinf}. Conversely, if
$\Fl^{\leq\ov{w}}\subseteq \Fl^{\leq\ov{u}}$, then $w_C\in \Fl^{\leq\ov{w}}\subseteq
\Fl^{\leq\ov{u}}$ (as in part~(a)), hence $w_C\in \Fl^{\leq_C u_C}$ for
every $C$. Therefore $w_C\leq_C u_C$ by \rp{semiinf}.

\smallskip

(c) Let $z$ be any element of $\Fl^{\leq\ov{w}}$, let $u\in\wt{W}$
be such that $z\in Iu$, and let $C\in\C{C}$ be such that $u\in C$.
We want to show that $u\leq w_C$, thus $z\in \Fl^{\leq w_C}$.

By \rl{ineq}(b), we get $z\in Iu\subseteq L(U_C)u$. On the other
hand, we have $z\in \Fl^{\leq\ov{w}}\subseteq \Fl^{\leq_C w_C}$. Therefore by
\rp{semiinf}, we get that $u\leq_C w_C$, which by \rco{order}(a)
implies that $u\leq w_C$.

\smallskip

(d) By construction, $z\in \Fl^{\leq\ov{u}}\cap L(U_C)u_C$ for all
$C\in\C{C}$, hence $\ov{u}$ is admissible by part~(a). Since $z\in
L(U_C)u_C\cap Z$, we get $u_C\in Z$ by \rl{invar}.

\smallskip

(e) follows immediately from part~(d).

\smallskip

(f) It is enough to show that for every $C\in\C{C}$, the preimage
$\pr^{-1}(\Gr^{\leq_C\mu_C})$ equals $\Fl^{\leq_C (\mu_C (w_{\on{st}})_C)}$. Using
\rp{semiinf}, we have to check that for every $\mu\in\La$ and
$u\in W$ we have $\mu\leq_C \mu_C$ if and only if $\mu
u\leq_C\mu_C (w_{\on{st}})_C$.

The ``only if'' assertion follows from \rco{ord}(b). Conversely, if
$\mu\leq_C \mu_C$, then $\mu u\leq_C \mu_C u$ by \rl{ord}(b). So
by Remark~\re{order}(a)(i), it is enough to show that $u\leq_C (w_{\on{st}})_C$.
Since $\ov{w}_{\on{st}}$ is admissible and $u=(w_{\on{st}})_{u(C_0)}$, the assertion
follows from \rl{adm}.
\end{proof}

\subsection{Proof of \rt{tuple}}

\begin{Emp}
 Let $m\in\B{N}$. Recall that  $w\in\wt{W}$ is called {\em
$m$-regular}, if $\pi(w)\in\La$ is $m$-regular, that is,
$|\langle\al,\pi(w)\rangle|\geq m$ for all $\al\in\Phi$. For each
$w\in\wt{W}$, we denote by $\wt{W}^{\leq w}$ the set of $w'\in\wt{W}$
such that $w'\leq w$.
\end{Emp}
The following result is a more precise version of \rt{tuple}.

\begin{Thm} \label{T:sch}
(a) For each $w\in\wt{W}$, there exists a unique admissible tuple
$\ov{w}$ such that the Schubert variety  $\Fl^{\leq w}$ equals
$\Fl^{\leq\ov{w}}$.

Moreover, $\ov{w}=\{w_C\}_C$ is characterized by condition that
$w_C$ is a unique maximal element of $\wt{W}^{\leq w}$ with
respect to the ordering $\leq_C$.

\smallskip

(b) Furthermore, there exists $r\in\B{N}$ such that for every $m\in
\B{N}$ and every $(m+r)$-regular $w\in\wt{W}$, the tuple $\ov{w}$
 is $m$-regular.
\end{Thm}

\begin{proof}
(a) Denote by $X(w)$ the closed ind-subscheme  $\bigcap_C(\bigcup_{w'\leq w}
\Fl^{\leq_C w'})\subseteq \Fl$ (compare \rco{adm}(c)), and we claim that
$X(w)$ equals $\Fl^{\leq w}$. Indeed, the inclusion $\Fl^{\leq w}\subseteq X(w)$
follows from \rco{adm}(e), while the opposite inclusion $X(w)\subseteq \Fl^{\leq w}$ follows
from identity $X(w)=\bigcup_{\ov{w}'\in (\wt{W}^{\leq w})^{\C{C}}}\Fl^{\leq
\ov{w}'}$ and \rco{adm}(c).

\smallskip

Next, since $\Fl^{\leq w}=\bigcup_{\ov{w}'\in (\wt{W}^{\leq
w})^{\C{C}}}\Fl^{\leq \ov{w}'}$ is irreducible, we see that there exists a
tuple $\ov{w}=\{w_C\}_C\in (\wt{W}^{\leq w})^{\C{C}}$ such that
$\Fl^{\leq w}=\Fl^{\leq\ov{w}}$.

\smallskip

Then for each $w'\leq w$ and $C\in\C{C}$, we have $w'\in \Fl^{\leq
w}\subseteq \Fl^{\leq_C w_C}$. Thus, by \rp{semiinf}, we have
$w'\leq_C w_C$, that is, $w_C$ is the biggest element of
$\wt{W}^{\leq w}$ with respect to ordering $\leq_C$. In
particular, for every other Weyl chamber $C'$ we have
$w_{C'}\leq_C w_C$. Thus by \rl{adm} we conclude that $\ov{w}$ is
admissible.

\smallskip

The uniqueness of $\ov{w}$ follows immediately from \rco{adm}(b).

\smallskip

(b) Choose any $\mu\in\La\cap C_0$, and let $r$ be the maximum of
the $2\langle\psi,\mu\rangle$'s, taken over $\psi\in\Psi_{C_0}$. We claim this
$r$ satisfies the required property, that is, for every $m\in\B{N}$ and
every $(m+r)$-regular $w\in\wt{W}$, the tuple $\ov{w}$ is
$m$-regular. In other words, we claim that $w_{u(C_0)}$ is
$(u(C_0),m)$-regular or, equivalently, that $u^{-1}w_{u(C_0)}$ is
$(C_0,m)$-regular for all $u\in W$.

\begin{Cl} \label{C:ineq}
Let $w\in\wt{W}$, and let $\ov{w}:=\{w_C\}_{C\in\C{C}}$
be the tuple from \rt{sch}(a).

\smallskip

(a) If $w=w_0w_+$, where $w_0\in W$ is the longest element, and
$w_+\in\wt{W}\cap C_0$, then $w_{u(C_0)}=uw_+$ for all $u\in W$.

\smallskip

(b) If $w\in Ww_+$ with $w_+\in\wt{W}\cap C_0$, then all $u\in W$
we have inequalities
\[
\mu^{-1}w_+\leq_{C_0}u^{-1}w_{u(C_0)}\leq_{C_0}w_+.
\]
\end{Cl}
\begin{proof}
(a) Fix $u\in W$. We will show that $w\leq w_0u^{-1}
w_{u(C_0)}\leq w$, which will imply that $w_0u^{-1}
w_{u(C_0)}=w_0w_+$, thus $w_{u(C_0)}=uw_+$.

Since $u\leq w_0$, we get $uw_+\leq w_0w_+=w$ (use Section~\re{border}(e)). Therefore by the characterization of $\ov{w}$,
given \rt{sch}(a) we get
$uw_+\leq_{u(C_0)} w_{u(C_0)}$. Hence
$w=w_0w_+\leq_{w_0(C_0)}w_0u^{-1} w_{u(C_0)}$ (by Remark~\re{order}(a)(ii)),
thus $w\leq w_0u^{-1} w_{u(C_0)}$ (by \rco{order}(a)). On the
other hand, $w_{u(C_0)}\leq w=w_0w_+$, thus using Section~\re{border}(e)
we conclude that $w_0u^{-1} w_{u(C_0)}\leq w$.

\smallskip

(b) By Remark~\re{order}(a)(ii), it is enough to show that
\[
u\mu^{-1}w_+\leq_{u(C_0)}w_{u(C_0)}\leq_{u(C_0)}uw_+.
\]
Consider element $w':=w_0w_+$. Then $w\leq w'$ (use Section~\re{border}(e)), thus we have $w_{u(C_0)}\leq w'$. Hence, $w_{u(C_0)}\leq_{u(C_0)} w'_{u(C_0)}$
(by the characterization of $w'_{u(C_0)}$, given in  \rt{sch}(a)), thus
$w_{u(C_0)}\leq_{u(C_0)}uw_+$ (by part~(a)). To show the second
inequality it is enough to show that $u\mu^{-1}w_+\leq w$. Since
$w$ and hence also $w_+$ is $(m+r)$-regular, our definition of
$r$ implies that $\mu^{-1}w_+\in\wt{W}\cap C_0$. Since $u\leq\mu$
(by Section~\re{border}(g)), and $l(w_+)=l(\mu)+l(\mu^{-1}w_+)$ (by
\rl{border}(c)), we conclude from Section~\re{border}(c) and part~(e) that
$u(\mu^{-1}w_+)\leq w_+\leq w$.
\end{proof}

Let us come back to the proof of the Theorem. By \rcl{ineq}(b)
and \rco{ord}(b), we have
\[
\mu^{-1}\pi(w_+)\leq_{C_0}\pi(u^{-1}w_{u(C_0)})\leq_{C_0}\pi(w_+).
\]
Hence we have $\pi(u^{-1}w_{u(C_0)})=\pi(w_+)-
\sum_{\al\in\Dt_{C_0}}m_{\al}\check{\al}$, such that $0\leq
m_{\al}\leq\langle\psi_{\al},\mu\rangle$, where
$\psi_{\al}\in\Psi_{C_0}$ is the fundamental weight corresponding
to $\al$ for each $\al\in\Dt_{C_0}$. In particular, for each
$\al\in\Dt_{C_0}$, we have
\[
\langle\al,\pi(u^{-1}w_{u(C_0)})\rangle\geq
\langle\al,\pi(w_+)\rangle-2m_{\al}\geq (m+r)-r=m,
\]
because $w_+$ is $(m+r)$-regular, and $2m_{\al}\leq
2\langle\psi_{\al},\mu\rangle\leq r$.
\end{proof}

\subsection{Technical lemmas}

\begin{Not} \label{E:psi} Fix $\psi\in\Psi$.

\smallskip
(a) Denote by
$P_{\psi}\supseteq T$ the parabolic subgroup of $G$ such that
$\Phi(P_{\psi},T)=\Phi(\psi)$ (see Section~\re{fund}(c)), by $M_{\psi}\supseteq
T$ the Levi subgroup of $P_{\psi}$, by $U_{\psi}\subseteq
P_{\psi}$ the unipotent radical, by $M^{\sc}_{\psi}$ the simply connected covering of the
derived (=commutator) group of $M_{\psi}$. Let $P_{\psi}\to M_{\psi}$ be the natural projection, and set
$P^{\sc}_{\psi}:=P_{\psi}\times_{M_{\psi}}M^{\sc}_{\psi}$.

\smallskip

(b) Note that we have a natural homomorphism $P^{\sc}_{\psi}\to P_{\psi}\subseteq G$, thus the loop group $L(P^{\sc}_{\psi})$ acts on $\Fl$. For every  $w\in \wt{W}$, we denote by $\Fl^{\leq_{\psi}w}\subseteq
\Fl$ the closure of the $L(P^{\sc}_{\psi})$-orbit
$L(P^{\sc}_{\psi})w\subseteq \Fl$.
\end{Not}

\begin{Lem} \label{L:semiinf}
(a) For $w',w''\in\wt{W}$ and $\psi\in\Psi$, we have
$w'\leq_{\psi} w''$ if and only if   $w'\in \Fl^{\leq_{\psi} w''}$.

\smallskip

(b) For $u\in\wt{W}$, $\psi\in\Psi$ and an admissible tuple
$\ov{w}\in\wt{W}^{\C{C}}$, we have $\Fl^{\leq\ov{w}}\subseteq
\Fl^{\leq_{\psi} u}$ if and only if $w_{\psi}\leq_{\psi} u$.
\end{Lem}

\begin{proof}
(a) Assume first that $w'\leq_{\psi} w''$, and we want to prove
that $L(P^{\sc}_{\psi})w'\subseteq \Fl$ is contained in the closure of
$L(P^{\sc}_{\psi})w''\subseteq \Fl$. By definition, we can assume that
$w'=s_{\wt{\beta}}w''<_{\wt{\beta}}w''$, where $\wt{\beta}=(\beta,m)$, and
$\langle\beta,\psi\rangle\geq 0$. Then there exists a Weyl chamber
$C\owns\psi$ such that $\beta\in\Phi_{C}$. Then $w'<_C w''$, hence
by \rp{semiinf}, $w'$ lies in the closure of $L(U_C)w''\subseteq
\Fl$. Since $L(U_C)\subseteq L(P^{\sc}_{\psi})$, the assertion
follows.

\smallskip

Conversely, assume that $w'$ belongs to the closure of
$L(P^{\sc}_{\psi})w''\subseteq \Fl$. Choose any Weyl chamber
$C\owns\psi$. Then $L(P^{\sc}_{\psi})w''$ is a union of orbits
$\bigcup_{w\in\wt{W}^{\psi}}L(U_C)ww''$. Therefore $w'$ belongs to
the closure of $L(U_C)ww''\subseteq \Fl$ for some $w\in
\wt{W}^{\psi}$. Hence by \rp{semiinf}, we get $w'\leq_C ww''$,
thus $w'\leq_{\psi}ww''$. On the other hand, since $w\in
\wt{W}^{\psi}$, we also get $ww''\leq_{\psi} w''$.

\smallskip

(b) Choose any Weyl chamber $C\owns\psi$. Then
$w_C\in\wt{W}^{\psi}w_{\psi}$, hence we have $w_{\psi}\leq_{\psi}
u$ if and only if $w_C\leq_{\psi} u$.

Assume first that $w_C\leq_{\psi} u$. Then by part~(a) we have
$\Fl^{\leq_{\psi} w_C}\subseteq \Fl^{\leq_{\psi} u}$. On the other
hand, we always have inclusions $\Fl^{\leq\ov{w}}\subseteq \Fl^{\leq_C
w_C}\subseteq \Fl^{\leq_{\psi} w_C}$, which imply that
$\Fl^{\leq\ov{w}}\subseteq \Fl^{\leq_{\psi} u}$. Conversely, since
$\ov{w}$ is admissible, we get $w_C\in \Fl^{\leq \ov{w}}$ by
\rl{adm}. Therefore if $\Fl^{\leq\ov{w}}\subseteq \Fl^{\leq_{\psi}
u}$, we get $w_C\in \Fl^{\leq_{\psi} u}$. Hence, by part~(a), we have
$w_C\leq_{\psi} u$.
\end{proof}

The remaining results of this subsection will be only  used in \rs{proof}.

\begin{Lem} \label{L:decomp}
(a) For all $\ov{w}',\ov{w}''\in\wt{W}^{\C{C}}$, there exist
admissible tuples $\ov{w}_1\ldots,\ov{w}_n$ from $\wt{W}^{\C{C}} $ such
that the reduced intersection $\Fl^{\leq\ov{w}'}\cap \Fl^{\leq\ov{w}''}$ equals
$\bigcup_{t=1}^n
\Fl^{\leq\ov{w}_t}$.

\smallskip

(b) For all $\ov{w}\in\wt{W}^{\C{C}}$, $\psi\in\Psi$ and
$u\in\wt{W}$, there exist admissible tuples
$\ov{w}_1\ldots,\ov{w}_n$ from $\wt{W}^{\C{C}} $ such that reduced intersection
$\Fl^{\leq\ov{w}}\cap \Fl^{\leq_{\psi} u}$ equals $\bigcup_{t=1}^n \Fl^{\leq\ov{w}_t}$.
\end{Lem}

\begin{proof}
We denote by $Z$ the reduced intersection $\Fl^{\leq\ov{w}'}\cap
\Fl^{\leq\ov{w}''}$ in the case (a), and
$\Fl^{\leq\ov{w}}\cap \Fl^{\leq_{\psi} u}$ in the case (b). Then, by
\rco{adm}(c), in both cases, $Z$ is a closed $T$-invariant
subscheme of $\Fl$ of finite type, thus the intersection
 $\wt{W}\cap Z$ is finite.

\smallskip

By \rco{adm}(d), each $z\in Z$ defines an admissible tuple
$\ov{u}=\ov{u}(z)\in\wt{W}^{\C{C}}$ satisfying $u_C\in \wt{W}\cap Z$
for each $C\in\C{C}$. It follows that the set of tuples
$\{\ov{u}(z)\}_{z\in Z}$ is finite, so it will suffice to show the equality
\begin{equation*} \label{Eq:int}
Z=\bigcup_{z\in Z} \Fl^{\leq \ov{u}(z)}.
\end{equation*}
One inclusion follows from the fact that $z\in \Fl^{\leq \ov{u}(z)}$
for every $z\in Z$. To show the converse, it is enough to show that
if $\ov{u}\in\wt{W}^{\C{C}}$ satisfies $u_C\in Z$ for all $C\in\C{C}$,
then $\Fl^{\leq\ov{u}}\subseteq Z$. Using definition of $Z$, it remains to
show the corresponding assertion in the cases $Z=\Fl^{\leq\ov{w}}$
and $Z=\Fl^{\leq_{\psi}u}$. In the first case, we have $u_C\in
\Fl^{\leq_C w_C}$, hence  $\Fl^{\leq_C u_C}\subseteq \Fl^{\leq_C w_C}$
for all $C\in\C{C}$, thus $\Fl^{\leq\ov{u}}\subseteq \Fl^{\leq\ov{w}}
$. In the second case, the assertion follows from \rl{semiinf}(b).
\end{proof}

We will need the following ``effective'' version of \rl{decomp}.

\begin{Lem} \label{L:int}
(a) There exists $r'\in\B{N}$ such that for every $m\in\B{N}$ and every two
$(m+r')$-regular admissible tuples
$\ov{w}',\ov{w}''\in\wt{W}^{\C{C}}$, there exists $m$-regular
admissible tuples $\ov{w}_1\ldots,\ov{w}_n\in\wt{W}^{\C{C}} $ such
that $\Fl^{\leq\ov{w}'}\cap \Fl^{\leq\ov{w}''}=\bigcup_{t=1}^n
\Fl^{\leq\ov{w}_t}$.

\smallskip

(b) There exists $r'\in\B{N}$ such that for every $m,d\in\B{N}$, every $(m+2d+r')$-regular admissible tuple
$\ov{w}\in\wt{W}^{\C{C}}$ and every $u\in\wt{W}$ satisfying
$\langle\psi,\pi(u)\rangle=\pi(\ov{w})_{\psi}-d$ there exists
$m$-regular admissible tuples
$\ov{w}_1\ldots,\ov{w}_n\in\wt{W}^{\C{C}} $ such that
\[
\Fl^{\leq\ov{w}}\cap \Fl^{\leq_{\psi} u}=\bigcup_{t=1}^n \Fl^{\leq\ov{w}_t}.
\]
\end{Lem}

The proof is based on the following two claims:

\begin{Cl} \label{C:int}
 (a) For every two quasi-admissible tuples $\ov{\mu}',\ov{\mu}''\in\La^{\C{C}}$,
there exists a unique maximal quasi-admissible tuple $\ov{\mu}\in\La^{\C{C}}$ such
that $\ov{\mu}\leq\ov{\mu}'$ and $\ov{\mu}\leq\ov{\mu}''$.
 Moreover, $\ov{\mu}$ is $m$-regular, if both $\ov{\mu}'$ and
$\ov{\mu}''$ are $m$-regular.

\smallskip

(b) For every $\psi\in\Psi$, $m,d\in\B{N}$ and every
$(m+2d)$-regular quasi-admissible tuple $\ov{\mu}\in \La^{\C{C}}$,
the tuple $\ov{\nu}:=\ov{\mu}-d\ov{e}_{\psi}$ is $m$-regular.
\end{Cl}

\begin{proof}
(a) Notice that $\ov{\mu}\leq\ov{\mu}'$ if and only if
$\mu_{\psi}\leq\mu'_{\psi}$ for all $\psi\in\Psi$. Thus a maximal
$\ov{\mu}$ satisfies $\mu_{\psi}=\min\{\mu'_{\psi},\mu''_{\psi}\}$
for all $\psi\in\Psi$. This shows the first assertion.

\smallskip

For the second one, choose $C\in\C{C}$, let $\al_1,\ldots,\al_r$
be the simple roots of $C$, and let $\psi_1,\ldots,\psi_r$ be the
corresponding fundamental weights. We want to show that
$\langle\al_j, \mu_C\rangle\geq m$ for all $j$. Without loss of
generality, we may assume that $\mu_{\psi_j}=\mu'_{\psi_j}$.
Recall that $\mu_C=\sum_{i=1}^r\mu_{\psi_i}\check{\al}_i$ and
$\mu'_C=\sum_{i=1}^r\mu'_{\psi_i}\check{\al}_i$.
 Since we have $\langle \al_j,\check{\al}_j\rangle=2>0, \mu_{\psi_j}=\mu'_{\psi_j}$ and
 $\langle \al_j,\check{\al}_i\rangle\leq 0, \mu_{\psi_i}\leq\mu'_{\psi_i}$ for all $i\neq j$,
 we conclude that $\langle\al_j,\mu_C\rangle\geq
\langle\al_j,\mu'_C\rangle\geq m$.

\smallskip

(b) Let $C, \al_i$ and $\psi_i$ be as in the proof of part~(a). Then
for every $j$, the pairing $\langle\al_j, \nu_C\rangle$ equals
$\langle\al_j, \mu_C\rangle-d\langle\al_j,\check{\al}_i\rangle\geq
(m+2d)-2d=m$, if $\psi=\psi_i$, and equals $\langle\al_j,
\mu_C\rangle\geq m+2d\geq m$, otherwise.
\end{proof}

\begin{Cl} \label{C:bound}
(a) There exists $r\in\B{N}$ such that for every $C\in\C{C}$, every root
$\al\in\Phi_C$ with corresponding fundamental weight $\psi\in\Psi_C$,
and every elements $w,w'\in\wt{W}$ with $w\leq_C w'$, we have either
$\check{\al}w\leq_{C}w'$ or
$\langle\psi,\pi(w')-\pi(w)\rangle\leq r$.

\smallskip

(b)  There exists $r\in\B{N}$ such that for every $\psi\in\Psi$,
$\al\in\Phi$ and $w,w'\in\wt{W}$ such that $w\leq_{\psi} w'$ and
$\langle\psi,\al\rangle=1$,
 we have either $\check{\al}w\leq_{\psi}w'$ or
 $\langle\psi,\pi(w')-\pi(w)\rangle\leq r$.
\end{Cl}

\begin{proof}
Since $W$ is finite, in both cases (a) and (b) it will be enough
to find $r$ to satisfy the condition for  $w\in\La u$ and $w'\in\La u'$, where
$u,u'\in W$ are fixed. Moreover, using Remark~\re{order}(a)(ii), we may
assume that $w'=u'$. Similarly, we fix $C\in\C{C}$, and
$\al\in\Phi_C$ with corresponding $\psi\in\Psi_C$.

In the case (a) we consider the set $S_C$  of all
$\mu\in\La$ such that $\mu u\leq_C u'$.
Then every $\mu\in S_C$ satisfies $\mu\leq_C 0$, hence
the set $S_C^{\max}$  of all maximal elements of $S_C$ with respect to the ordering $\leq_C$ is finite and non-empty.
We take $r\in\B{N}$ to be the maximum of all $-\langle\psi,\mu\rangle$ taken over all $\mu\in S_{C}^{\max}$.

In the case (b) we consider the set $S_{\psi}$ of all
$\mu'\in\La$ such that $\mu' u\leq_{\psi} u'$. Then every $\mu'\in
S_{\psi}$ satisfies $\mu'\leq_{\psi}0$, hence the set $S_{\psi}^{\max}$ of all maximal elements of $S_{\psi}$ with respect to the ordering $\leq_{\psi}$ is a finite and non-empty
union of cosets of $\La^{\psi}:=\{\mu\in\La\,|\,\lan\psi,\mu\ran=0\}$.
We take $r\in\B{N}$ to be the maximum of all $-\langle\psi,\mu'\rangle$, taken
over all  $\mu'\in S^{\max}_{\psi}$.

Then in both cases $r$ satisfies the required property. Indeed,
assume that $\mu\in S_C$ (resp. $\mu\in S_{\psi}$)
while $\check{\al}\mu\notin S_C$, (resp. $\check{\al}\mu\notin
S_C$), and we want to check that $\langle\psi,\mu\rangle\geq -r$.
Choose any $\mu'\in S_C^{\max}$ (resp.  $\mu'\in S_{\psi}^{\max}$) be such
that $\mu'\geq_C\mu$ (resp.  $\mu'\geq_{\psi}\mu$.) Then $\mu'-\mu=\sum_{\beta\in\Dt_C}m_{\beta}\check{\beta}$
and $m_{\al}\geq 0$. Since $\check{\al}\mu\notin S_C$ (resp. $\check{\al}\mu\notin S_{\psi}$), we have $m_{\al}=0$,
thus $\langle\psi,\mu\rangle=\langle\psi,\mu'\rangle\geq -r$.
\end{proof}

Now we are ready to prove \rl{int}.

\begin{Emp}
\begin{proof}[Proof of \rl{int}]
(a) Let $r\in\B{N}$ be as in \rcl{bound}(a). We will show that
$r':=2r$ satisfies the required property. Let
$\ov{w}',\ov{w}''\in\wt{W}^{\C{C}}$ be $(m+r')$-regular admissible
tuples. Then, by \rl{decomp}, there exist admissible tuples
$\ov{w}_1\ldots,\ov{w}_n\in\wt{W}^{\C{C}} $ such that
$\Fl^{\leq\ov{w}'}\cap \Fl^{\leq\ov{w}''}=\bigcup_{t=1}^n \Fl^{\leq\ov{w}_t}$.

\smallskip

Using \rco{adm}(b), one can assume that each  $\ov{w}_t$
is a maximal admissible tuple, satisfying
$\ov{w}_t\leq\ov{w}',\ov{w}''$, and we have to show that each $\ov{w}_t$ is
$m$-regular.

\smallskip

 Let $\ov{w}\in\wt{W}^{\C{C}}$ be a maximal quasi-admissible tuple,
satisfying $\ov{w}\leq\ov{w}',\ov{w}''$.  It is enough to show that
such a $\ov{w}$ is $m$-regular. Indeed, \rl{ord1}(a) then would imply
that $\ov{w}$ is admissible.

Set $\ov{\mu}':=\pi(\ov{w}')$, and $\ov{\mu}'':=\pi(\ov{w}'')$,
and let $\ov{\mu}$ be the maximal tuple such that $\ov{\mu}\leq
\ov{\mu}'$ and $\ov{\mu}\leq\ov{\mu}''$. Then  $\ov{\mu}$ is
$(m+r')$-regular by \rcl{int}(a), and  $\pi(\ov{w})\leq\ov{\mu}$
by \rco{ord}(b).

\smallskip

It is enough to show that $\pi(\ov{w})_{\psi}\geq \mu_{\psi}-r$
for every $\psi\in\Psi$. Indeed, if this is shown, then for every
$C\in\C{C}$ with simple roots $\al_1,\ldots,\al_r$, we have
$\pi(w_C)=\mu_C-\sum_i r_i\check{\al}_i$ and $0\leq r_i\leq r$ for
all $i$. Then $\langle\al_i,\pi(w_C)\rangle\geq
\langle\al_i,\mu_C\rangle-2r_i\geq (m+2r)-2r=m$. Thus $\ov{w}$ is
$m$-regular.

\smallskip

Assume that there exists $\psi\in\Psi$ such that
$\pi(\ov{w})_{\psi}<\mu_{\psi}-r$. Consider the quasi-admissible
tuple $\ov{e}_{\psi}$ defined by
$(\ov{e}_{\psi})_{\psi'}:=\delta_{\psi,\psi'}$ (see Section~\re{qadm}(c)).
Then the quasi-admissible  tuple
$\ov{e}_{\psi}\ov{w}$ (see Section~\re{qadm}(f)) satisfies identities
$(\ov{e}_{\psi}\ov{w})_C=w_C$ if $\psi\notin \Psi_C$, and
$(\ov{e}_{\psi}\ov{w})_C=\check{\al}\ov{w}_C$ if $\psi\in \Psi_C$ and $\al\in\Dt_C$
corresponds to $\psi$.

\smallskip

Since $\ov{w}\leq\ov{w}'$ and $\ov{w}\leq\ov{w}''$,
the assumption $\pi(\ov{w})_{\psi}<\mu_{\psi}-r$ together with
\rcl{bound}(a) implies that  $\ov{e}_{\psi}\ov{w}\leq\ov{w}'$ and
$\ov{e}_{\psi}\ov{w}\leq\ov{w}''$. Since $\ov{w}<\ov{e}_{\psi}\ov{w}$, this
contradicts the maximality of $\ov{w}$.

\smallskip

(b) The proof is similar to that of part~(a). Let $r\in\B{N}$ to
satisfy both \rcl{bound}(a),(b), and set $r':=2r$. Assume that
$\ov{w}\in\wt{W}^{\C{C}}$ is $(m+2d+r')$-regular,
$u\in\wt{W}$ satisfies
$\langle\psi,\pi(u)\rangle=\pi(\ov{w})_{\psi}-d$, and let
$\ov{w}'$ be a maximal quasi-admissible tuple satisfying
$\ov{w}'\leq\ov{w}$ and $\ov{w}'_{\psi}\leq_{\psi}u$.

\smallskip

Using \rl{decomp}(b) and \rl{semiinf}(b), and arguing as in part~(a),
it is enough to show that $\pi(\ov{w})_{\psi'}\geq \mu_{\psi'}-r$
for every $\psi'\in\Psi$.

\smallskip

Assume that there exists $\psi'\in\Psi$ such that
$\pi(\ov{w})_{\psi'}<\mu_{\psi'}-r$, and let $\ov{e}_{\psi'}\ov{w}$ be as
in part~(a). Again to get a contradiction, it is enough to show that
$\ov{e}_{\psi'}\ov{w}\leq\ov{w}'$ and $(\ov{e}_{\psi'}\ov{w})_{\psi}\leq_{\psi} u$.
The proof of the first inequality is identical to that of part~(a). Next, if
$\psi'\neq\psi$, then
$(\ov{e}_{\psi'}\ov{w})_{\psi}=\ov{w}_{\psi}\leq_{\psi} u$ by assumption.
Finally, if $\psi'=\psi$, the inequality
$(\ov{e}_{\psi'}\ov{w})_{\psi}\leq_{\psi} u$ follows from \rcl{bound}(b).
\end{proof}
\end{Emp}

\begin{Lem} \label{L:finite}
There exists $r\in\B{N}$ such that for every $m\in\B{N}$
 and every $(m+r)$-regular $\ov{w}\in\wt{W}^{\C{C}}$, there exists an $m$-regular
 $\ov{x}\in\La^{\C{C}}$ such that $\Fl^{\leq'\ov{x}}\subseteq \Fl^{\leq\ov{w}}$.
\end{Lem}

\begin{proof}
Choose any $\mu\in\La\cap C_0$, let $\ov{\mu}\in\La^{\C{C}}$ be the
admissible tuple defined by $\mu_{u(C_0)}:=u(\mu)$ (see Section~\re{exam}(a)), and let $r$ be
the maximum of the $\langle\al,\mu\rangle$'s, where $\al$ runs over all of $\Dt_{C_0}$. We claim that this $r$ satisfies the
required property.

\smallskip

Namely, to every  $(m+r)$-regular admissible tuple
$\ov{w}\in\wt{W}^{\C{C}}$, we associate a quasi-admissible tuple
$\ov{x}:=\ov{\mu}^{-1}\pi(\ov{w})$ (see Section~\re{qadm}(e)). We claim that
$\ov{x}$ is $m$-regular, and $\Fl^{\leq'\ov{x}}\subseteq \Fl^{\leq
\ov{w}}$.

\smallskip

To show that $\ov{x}$ is $m$-regular, we note that for
every $u\in W$, $C=u(C_0)\in\C{C}$ and $\al\in\Dt_C$, we have
$\langle\al,x_C\rangle=\langle\al,\pi(w_C)\rangle-\langle\al,u(\mu)\rangle\geq
(m+r)-r=m$.

\smallskip

Next we observe that $\Fl^{\leq'\ov{x}}=\Fl^{\leq \ov{x}\cdot
\ov{w}_{\on{st}}}$ (use \rco{adm}(f)). So it remains to show that
$\ov{x}\cdot\ov{w}_{\on{st}}\leq\ov{w}$ or, what is the same, $x_C
u\leq_C w_C$ for each $C=u(C_0)\in\C{C}$. Unwinding the definitions and using Section~\re{order}(a), it is enough to show that for every
$u\in W$ we have $1\leq_{C_0}\mu u$. By \rco{order}, it remains to show that $\mu u\in \wt{W}\cap C_0$,
that is, for every $\al\in\Phi_{C_0}$ we have $(\mu u)^{-1}(\al)>0$. But $(\mu
u)^{-1}(\al)=u^{-1}(\mu^{-1}(\al))=(u^{-1}(\al),\langle\al,\mu\rangle)>0$,
because $\langle\al,\mu\rangle>0$.
\end{proof}

\begin{Lem} \label{L:seq}
There exists $r\in\B{N}$ such that for every $m\in\B{Z}$ and every
$(m+r)$-regular  tuple $\ov{x}\in\La^{\C{C}}$
there exists a sequence $\ov{x}=\ov{x}_0\leq\ov{x}_1\leq\ldots$  in $\La^{\C{C}}$
such  that sequence $\{(x_i)_{\psi}\}_i$ tends to infinity for all $\psi\in\Psi$,
each  $\ov{x}_i$ is $m$-regular, and $\ov{x}_i=\ov{x}_{i-1}+\ov{e}_{\psi_{i}}$ for some $\psi_i\in\Psi$ and all $i$.
\end{Lem}
\begin{proof}
Choose $\mu\in\La\cap C_0$, let $\ov{\mu}\in\La^{\C{C}}$ be the
tuple $\mu_{u(C_0)}:=u(\mu)$ from Section~\re{exam}(a). Then $\ov{\mu}$ is a
regular and admissible, and let $\ov{y}\in\B{N}^{\Psi}$ be the corresponding tuple (see Section~\re{qadm}(b) and \rl{ord1}(b)).
Choose a sequence $\ov{y}_0=0,\ov{y}_1,\ldots,\ov{y}_n=\ov{y}$ in $\B{N}^{\Psi}$
such that $\ov{y}_i-\ov{y}_{i-1}=\ov{e}_{\psi_i}$ for all $i$ and some
$\psi_i\in\Psi$, and continue it to all $i$ by the rule
$ \ov{y}_{i+n}:=\ov{y}_i+\ov{y}$.

\smallskip

Define $r$ to be the maximum of the $-\langle (y_i)_C,\al\rangle$'s,
taken over $i=1,\ldots,n$, $C\in\C{C}$ and $\al\in\Dt_C$.
Then the sequence $\ov{x}_i:=\ov{x}+\ov{y}_i$ satisfies the
required property.
\end{proof}

\subsection{Stratification of the affine flag variety}
\begin{Not} \label{E:redgroups}
(a) Let $k$, $K$ and $\C{O}$ be as in Section~\re{groups}, let $G$ be a
connected reductive group over $k$, and let $T\subseteq G$ be a maximal torus.

\smallskip

(b) Let $G^{\sc}$ be the simply connected covering of the derived group of $G$, and let
$T_{G^{\sc}}\subseteq G^{\sc}$ be the corresponding maximal torus, that is, the pullback of $T\subseteq G$. Let $\Phi$ be the root system $\Phi(G,T)=\Phi(G^{\sc},T_{G^{\sc}})$, let $\Psi$ the set of dominant weights, and let $\wt{W}$ the affine Weyl group of $G^{\sc}$.

\smallskip

(c) We choose an Iwahori subgroup $I\subseteq LG$ as in Section~\re{notation}, set $I^{\sc}:=I\cap L(G^{\sc})\subseteq L(G^{\sc})$, and let
$\Fl=\Fl_{G^{\sc}}:=L(G^{\sc})/I^{\sc}$ the affine flag variety of $G^{\sc}$.
\end{Not}

\begin{Not} \label{E:redgp}
In the situation of Section~\re{redgroups}, fix $\psi\in\Psi\subseteq X^*(T_{G^{\sc}})$.

\smallskip

(a) Let $P_{\psi}, M_{\psi}, U_{\psi}, M^{\sc}_{\psi}$ and $P^{\sc}_{\psi}$ be as in Section~\re{psi}(b).
Notice that groups $M^{\sc}_{\psi}, U_{\psi}$ and $P^{\sc}_{\psi}$ would not change if we replace group $G$ by $G^{\sc}$.

\smallskip

(b) As in Section~\re{psi}(b), we have a natural homomorphism $P^{\sc}_{\psi}\to G^{\sc}$, thus the loop group $L(P^{\sc}_{\psi})$ acts on $\Fl$.
For every $w\in \wt{W}$, we denote by $\Fl^{\leq_{\psi}w}\subseteq
\Fl$ the closure of the $L(P^{\sc}_{\psi})$-orbit $L(P^{\sc}_{\psi})w\subseteq \Fl$.

\smallskip

(c) As in Section~\re{groups}, we have an equality $\La=X_*(T^{\sc})$. As in Section~\re{fund}(b), the coweight $\check{\psi}$ belongs to $\La_{\B{Q}}$. We denote by $T_{\psi}\subseteq T_{G^{\sc}}$ the one-dimensional subtorus such that $X_*(T_{\psi})\subseteq \La$
equals $\B{Z}\check{\psi}\cap \La\subseteq\La_{\B{Q}}$.

\smallskip

(d) Alternatively, $T_{\psi}$ can be defined as the connected center of the Levi subgroup $(M_{\psi})_{G^{\sc}}$ of $G^{\sc}$, where
$(M_{\psi})_{G^{\sc}}\subseteq G^{\sc}$ is the pullback of $M_{\psi}\subseteq G$.
\end{Not}

\begin{Emp} \label{E:strata}

{\bf Stratification.}

\smallskip

(a) For each $\nu\in \wt{W}_{\psi}$, we set $Z_{\nu}:=\Fl^{\leq_{\psi}\nu}\sm\bigcup_{\nu'<_{\psi}\nu}\Fl^{\leq_{\psi}\nu'}$.
Then each $Z_{\nu}\subseteq \Fl$ is a reduced locally closed $L(P^{\sc}_{\psi})$-invariant ind-subscheme.
Moreover, since $\wt{W}_{\psi}$ is a set of representatives of the set of cosets $\wt{W}^{\psi}\bs\wt{W}$ (see Section~\re{fund}(d)), the set
$\{Z_{\nu}\}_{\nu\in\wt{W}_{\psi}}$ forms a stratification of $\Fl$.
\smallskip

(b) For each $\nu\in \wt{W}_{\psi}$, we set $I_{\nu}:=\nu I\nu^{-1}\subseteq LG$, $I_{P_{\psi},\nu}:=I_{\nu}\cap L(P_{\psi})\subseteq L(P_{\psi})$,  $I_{M_{\psi},\nu}:=I_{\nu}\cap L(M_{\psi})\subseteq L(M_{\psi})$ and $I_{U_{\psi},\nu}:=I_{\nu}\cap L(U_{\psi})\subseteq L(U_{\psi})$, and let $I_{P^{\sc}_{\psi},\nu}\subseteq L(P^{\sc}_{\psi})$ and $I_{M^{\sc}_{\psi},\nu}\subseteq L(M^{\sc}_{\psi})$ be the preimages of $I_{P_{\psi},\nu}$ and $I_{M_{\psi},\nu}\subseteq L(M_{\psi})$, respectively. We also set $\Fl_{P^{\sc}_{\psi},\nu}:=L(P^{\sc}_{\psi})/I_{P^{\sc}_{\psi},\nu}$ and $\Fl_{M^{\sc}_{\psi},\nu}:=L(M^{\sc}_{\psi})/I_{M^{\sc}_{\psi},\nu}$.

\smallskip

(c) Notice that isomorphism $U_{\psi}\times M_{\psi}\isom P_{\psi}:(u,m)\mapsto um$ induces isomorphisms $U_{\psi}\times M^{\sc}_{\psi}\isom P^{\sc}_{\psi}$ and $I_{U_{\psi},\nu}\times I_{M^{\sc}_{\psi},\nu}\isom I_{P^{\sc}_{\psi},\nu}$, and that the embedding and the projection $M^{\sc}_{\psi}\to P^{\sc}_{\psi}\to M^{\sc}_{\psi}$ induce morphisms $\Fl_{M^{\sc}_{\psi},\nu}\overset{i_{\psi,\nu}}{\lra} \Fl_{P^{\sc}_{\psi},\nu}\overset{p_{\psi,\nu}}{\lra}\Fl_{M^{\sc}_{\psi},\nu}$.

\smallskip

(d) By \rl{semiinf}(a), each $Z_{\nu}\subseteq\Fl$ is an $L(P^{\sc}_{\psi})$-orbit of $\nu\in\Fl$.
Moreover, $L(P^{\sc}_{\psi})\simeq L(M^{\sc}_{\psi})\times L(U_{\psi})$ is reduced (see \cite{BD}, if $k$ is of characteristic zero, and \cite{Fa} in general), so the morphism $[h]\mapsto h\nu$ induces an isomorphism $\iota_{\nu}:\Fl_{P^{\sc}_{\psi},\nu}\isom Z_{\nu}$.

\smallskip

(e) Since $T_{G^{\sc}}$ normalizes $P^{\sc}_{\psi}$ and fixes $\nu\in \Fl$, the
orbit $Z_{\nu}\subseteq \Fl$ is $T_{G^{\sc}}$-invariant, hence $T_{\psi}$-equivariant. Furthermore,
the isomorphism $\iota_{\nu}$ of part~(d) identifies the $T_{\psi}$-action on $Z_{\nu}$ with
the $T_{\psi}$-action on $\Fl_{P^{\sc}_{\psi},\nu}$ given by the formula $t[um]=[tut^{-1}m]$ for
$u\in L(U_{\psi})$ and $m\in L(M^{\sc}_{\psi})$. In particular, isomorphism  $\iota_{\nu}$ induces an isomorphism
$\iota^{T_{\psi}}_{\nu}:\Fl_{M^{\sc}_{\psi},\nu}\isom Z^{T_{\psi}}_{\nu}:[m]\mapsto m\nu$, where $Z_{\nu}^{T_{\psi}}$ denotes the locus of $T_{\psi}$-fixed points.

\smallskip

(f) As in Section~\re{notation}(b), the choice of the Weyl chamber $C_0^{\psi}$ of $\Phi^{\psi}$ (see Section~\re{fund}(d)) gives rise to an Iwahori subgroup $I_{M^{\sc}_{\psi}}\subseteq L(M^{\sc}_{\psi})$, hence the affine flag variety $\Fl_{M^{\sc}_{\psi}}:=L(M^{\sc}_{\psi})/I_{M^{\sc}_{\psi}}$. Moreover, since $\nu\in \wt{W}_{\psi}$, we have an inclusion $I_{M^{\sc}_{\psi}}\subseteq I_{\nu}$. Furthermore, we have an equality $I_{M^{\sc}_{\psi}}=I_{M^{\sc}_{\psi},\nu}$ (use, for example, \cite[Lemma~B.1.5]{BV}), thus an isomorphism  $\Fl_{M^{\sc}_{\psi}}\isom \Fl_{M^{\sc}_{\psi},\nu}$.

\smallskip

(g) Note that the affine flag variety $\Fl_{M^{\sc}_{\psi},\nu}$ of part~(c) is ind-proper. Hence $\Fl_{M^{\sc}_{\psi},\nu}\subseteq\Fl$ is closed, so
it follows from part~(e) that each $Z^{T_{\psi}}_{\nu}\subseteq \Fl^{T_{\psi}}$ is open and closed.
\end{Emp}




\begin{Emp} \label{E:retract}
{\bf Retraction.}
Let $Y$ be an ind-scheme, and let $Z\subseteq Y$ be a locally closed ind-subscheme. A morphism $p:Y\to Z$ is
called a {\em retraction}, if the restriction $f|_Y$ is the identity.
%
%
\end{Emp}

\begin{Lem} \label{L:strata}
For every element $\nu\in\wt{W}_{\psi}$, there is a unique $T_{\psi}$-equivariant retraction $p_{\nu}:Z_{\nu}\to Z_{\nu}^{T_{\psi}}$.
Moreover, under an isomorphisms 
of Sections~\re{strata}(d)-(f), the retraction $p_{\nu}$ corresponds to the projection
$p_{\psi,\nu}:\Fl_{P^{\sc}_{\psi},\nu}\to\Fl_{M^{\sc}_{\psi}}:[um]\mapsto [m]$.
\end{Lem}

\begin{proof}
To see the existence of a retraction and its relation to $p_{\psi,\nu}$, we note that $\iota_{\nu}$ induces an isomorphism $\Fl_{M^{\sc}_{\psi},\nu}=\Fl^{T_{\psi}}_{P^{\sc}_{\psi},\nu}\isom Z^{T_{\psi}}_{\nu}$, and that
$p_{\psi,\nu}:\Fl_{P^{\sc}_{\psi},\nu}\to\Fl_{M^{\sc}_{\psi},\nu}$ is a $T^{\psi}$-equivariant retract. To see the uniqueness of the retract, we note that  for every $S$-point $\eta:S\to \Fl_{P^{\sc}_{\psi},\nu}$, the morphism
$\eta_{\B{G}_m}: \B{G}_m\times S\to \Fl_{P^{\sc}_{\psi},\nu}$, defined by $(a,x)\mapsto \psi(a)\nu(s)$ extends uniquely to the morphism $\eta_{\B{A}^1}:\B{A}^1\times S\to \Fl_{P^{\sc}_{\psi},\nu}$, and we have an equality $p_{\psi,\nu}(\eta)=\eta_{\B{A}^1}|_{\{0\}\times S}$.
\end{proof}

\section{Affine Springer fibers}

\subsection{Geometric properties} Assume that we are in the situation Section~\re{redgroups}.

\begin{Emp} \label{E:affsprfib}
{\bf Set-up.} (a) Let $\gm\in
G(K)$ be a compact regular semi-simple element, let $G^{\sc}_{\gm}\subseteq G^{\sc}$ be
the centralizer of $\gm$, let $S_{\gm}\subseteq G^{\sc}_{\gm}$ be the maximal split
torus of $G^{\sc}_{\gm}$ (over $K$), and let $\La_{\gm}:=X_*(S_{\gm})$ be the group of cocharacters.
The map $\mu\mapsto \mu(t)$ identifies $\La_{\gm}$ with a subgroup of $S_{\gm}(K)$.

\smallskip

(b) Let $\Fl_{\gm}\subseteq \Fl$ be the affine Springer fiber. Explicitly, $\Fl_{\gm}$ consists of cosets $gI^{\sc}\in L(G^{\sc})/I^{\sc}$ such that $g^{-1}\gm g\in I$. Then the group $\La_{\gm}$ acts on $\Fl_{\gm}$. Moreover, it is known that there exists a closed subscheme of finite type $Y\subseteq \Fl_{\gm}$ such that $\Fl_{\gm}=\La_{\gm}(Y)$.

\smallskip

(c) For every ind-subscheme $Z\subseteq\Fl$, we set $Z_{\gm}:=Z\cap\Fl_{\gm}$.

\smallskip

(d) {\bf Main assumption:} We always assume that we have an inclusion $S_{\gm}\subseteq T_{G^{\sc}}$, hence an inclusion $\La_{\gm}\subseteq\La=X_*(T_{G^{\sc}})$.

\end{Emp}

\begin{Emp} \label{E:remcent}
{\bf Remarks.} Suppose that we are in the situation of Section~\re{affsprfib}(a).

\smallskip

(a) By \cite{St}, the centralizer $G^{\sc}_{\gm}\subseteq G^{\sc}$ is connected, thus is a maximal torus.

\smallskip

(b) The centralizer $G^{\sc}_{S_{\gm}}\subseteq G^{\sc}$ is a Levi subgroup, and $S_{\gm}$ is the connected center of $G^{\sc}_{S_{\gm}}$. Indeed, the centralizer $G^{\sc}_{S_{\gm}}$ is split over $K$ because $G^{\sc}$ and $S_{\gm}$ are split over $K$, therefore the connected center $Z(G^{\sc}_{S_{\gm}})^0$ of $G^{\sc}_{S_{\gm}}$ is split over $K$ as well. Moreover, since $G^{\sc}_{\gm}$ is a maximal torus of $G^{\sc}$, it is a maximal torus of $G^{\sc}_{S_{\gm}}$, hence contains  $Z(G^{\sc}_{S_{\gm}})^0$. Therefore the assertion follows from the assumption that $S_{\gm}\subseteq G^{\sc}_{\gm}$ is the maximal split torus.
\end{Emp}

\begin{Emp} \label{E:remgood}
{\bf Observations.} Fix $\psi\in\Psi$.

\smallskip

(a) Using Section~\re{redgp}(d), an inclusion  $(G_{\gm})^0\subseteq M_{\psi}$ is equivalent to the inclusion
$G^{\sc}_{\gm}\subseteq (M_{\psi})_{G^{\sc}}$, thus to an inclusion $T_{\psi}\subseteq G^{\sc}_{\gm}$,
hence to an inclusion $T_{\psi}\subseteq S_{\gm}$.

\smallskip

(b) Set $(\La_{\gm})_{\B{Q}}:=\La_{\gm}\otimes_{\B{Z}}\B{Q}$. By Section~\re{redgp}(c), an inclusion $T_{\psi}\subseteq S_{\gm}$ holds if and only if  $\check{\psi}\in(\La_{\gm})_{\B{Q}}$.

\smallskip

(c) It follows from parts~(a) and (b), that if $\check{\psi}\in(\La_{\gm})_{\B{Q}}$, then
element $\gm$ belongs to $(G_{\gm})^0(K)\subseteq M_{\psi}(K)\subseteq P_{\psi}(K)$.

\smallskip

(d) It follows from Section~\re{remcent}(b) that if $\check{\psi}\notin(\La_{\gm})_{\B{Q}}$,
then there exists a root $\al\in\Phi$ such that $\al\in (\La_{\gm})^{\perp}$, but $\lan\al,\check{\psi}\ran\neq 0$.

\end{Emp}

\begin{Not} \label{E:goodpos}
Assume that $\psi\in\Psi$ satisfies $\check{\psi}\in(\La_{\gm})_{\B{Q}}$.

\smallskip

(a) By Section~\re{remgood}(c), we have $\gm\in M_{\psi}(K)\subseteq P_{\psi}(K)$, thus we can consider the affine Springer fibers $\Fl_{P^{\sc}_{\psi},\nu,\gm}\subseteq \Fl_{P^{\sc}_{\psi},\nu}$ and  $\Fl_{M^{\sc}_{\psi},\gm}\subseteq \Fl_{M^{\sc}_{\psi},\nu}$.
Explicitly, $\Fl_{P^{\sc}_{\psi},\nu,\gm}$ (resp. $\Fl_{M^{\sc}_{\psi},\nu,\gm}$) consists of all elements $gI_{P^{\sc}_{\psi},\nu}\in L(P^{\sc}_{\psi})/I_{P^{\sc}_{\psi},\nu}$ (resp. $gI_{M^{\sc}_{\psi}}\in L(M^{\sc}_{\psi})/I_{M^{\sc}_{\psi}}$) such that
$g^{-1}\gm g\in I_{P_{\psi},\nu}$ (resp. $g^{-1}\gm g\in I_{M_{\psi}}$).

\smallskip

(b) By construction, the isomorphism  $\iota_{\nu}:\Fl_{P^{\sc}_{\psi},\nu}\isom Z_{\nu}$ from Section~\re{strata}(e) restricts to isomorphisms
$\iota_{\nu,\gm}:\Fl_{P^{\sc}_{\psi},\nu,\gm}\isom Z_{\nu,\gm}$ and  $\iota^{T_{\psi}}_{\nu,\gm}:\Fl_{M^{\sc}_{\psi}}\isom Z^{T_{\psi}}_{\nu,\gm}$.
\end{Not}





\begin{Emp} \label{E:affine bundle}
{\bf Affine bundle.} A morphism $f:X\to Y$ of schemes is called an {\em
affine bundle} if locally \'etale on $Y$ it is isomorphic to the
projection $Y\times \B{A}^n\to Y$ and all transition maps are
affine.
\end{Emp}

\begin{Prop} \label{P:strata}
Assume that we are in the situation of Section~\re{goodpos}. Then for every  $\nu\in\wt{W}_{\psi}$ the
$T_{\psi}$-equivariant retraction $p_{\nu}:Z_{\nu}\to Z_{\nu}^{T_{\psi}}$ of \rl{strata}
induces a retraction $p_{\nu,\gm}:Z_{\nu,\gm}\to Z_{\nu,\gm}^{T_{\psi}}$. Furthermore, $p_{\nu,\gm}$ is a
composition of affine bundles.

\end{Prop}

\begin{proof} To make the argument more structural, we will divide it into steps.

\smallskip

\noindent{\bf Step 1.} By \rl{strata} and the observations of Section~\re{goodpos} it suffices to show that the projection
$p_{\psi,\nu}:\Fl_{P^{\sc}_{\psi},\nu}\to\Fl_{M^{\sc}_{\psi}}$ restricts to the projection
\[
p_{\psi,\nu,\gm}:\Fl_{P^{\sc}_{\psi},\nu,\gm}\to\Fl_{M^{\sc}_{\psi},\gm},
\]
and that $p_{\psi,\nu,\gm}$ is a composition of affine bundles.

\smallskip

\noindent{\bf Step 2.} Let $U_{\psi}=U_0\supseteq U_1\supseteq \ldots \supseteq U_{n-1}\supseteq
U_n=\{1\}$ be the lower central series of $U_{\psi}$. Then each
$U_i$ is a normal subgroup of $P_{\psi}$, and we set
$P_i:=P_{\psi}/U_i$ and $P^{\sc}_i:=P^{\sc}_{\psi}/U_i$. In particular, $P_0=M_{\psi}$ and
$P_n=P_{\psi}$.

\smallskip

For every $i=0,\ldots,n$, let $\gm_i\in L(P_i)$ be the image of $\gm\in L(P_{\psi})$, and denote by $I_{P_i,\nu}\subseteq L(P_i)$ (resp. $I_{P^{\sc}_i,\nu}\subseteq L(P^{\sc}_i)$) the image of $I_{P_{\psi},\nu}$ (resp. $I_{P^{\sc}_{\psi},\nu}$). We set $\Fl_{P^{\sc}_i,\nu}:=L(P^{\sc}_i)/I_{P^{\sc}_i,\nu}$, and let $\Fl_{P^{\sc}_{i},\nu,\gm}\subseteq \Fl_{P^{\sc}_i,\nu}$ be the
corresponding affine Springer fiber, that is, the collection of all
$g\in L(P^{\sc}_i)/I_{P^{\sc}_i,\nu}$ such that $g^{-1}\gm_i g\in I_{P_i,\nu}$.

\smallskip

For every $i=0,\ldots,n-1$, we have a natural projection
\[
p_{i,\gm}:\Fl_{P^{\sc}_{i+1},\nu,\gm_{i+1}}\to\Fl_{P^{\sc}_{i},\nu,\gm_i},
\]
and it remains to show that each $p_{i,\gm}$ is an affine bundle.

\smallskip

\noindent{\bf Step 3.} Let $\wt{\Fl}_{P^{\sc}_{i},\nu,\gm_i}\subseteq L(P^{\sc}_i)$ be the preimage of
$\Fl_{P^{\sc}_{i},\nu,\gm_i}\subseteq \Fl_{P^{\sc}_{i},\nu}$ under the natural projection
$L(P^{\sc}_i)\to L(P^{\sc}_i)/I_{P^{\sc}_i,\nu}$, and set
\[
\wt{\Fl}'_{P^{\sc}_{i+1},\nu,\gm_{i+1}}:=\wt{\Fl}_{P^{\sc}_{i},\nu,\gm_i}
\times_{\Fl_{P^{\sc}_{i},\nu,\gm_i}} \Fl_{P^{\sc}_{i+1},\nu,\gm_{i+1}}.
\]
It is enough to show that each projection
$\wt{\Fl}'_{P^{\sc}_{i+1},\nu,\gm_{i+1}}\to\wt{\Fl}_{P^{\sc}_{i},\nu,\gm_{i}}$ is an
affine bundle.

\smallskip

We set $\ov{U}_{i}:=U_{i}/U_{i+1}$. Then $\ov{U}_i\subseteq
P_{i+1}=P_{\psi}/U_{i+1}$ is a normal subgroup, and we have
$P_i\cong P_{i+1}/\ov{U}_i$. Set $I_{\ov{U}_i,\nu}:=I_{P_{i+1},\nu}\cap L(\ov{U}_i)$. Then
$\wt{\Fl}'_{P^{\sc}_{i+1},\nu,\gm_{i+1}}$ can be identified with the locus
of all $g\in L(P^{\sc}_{i+1})/I_{\ov{U}_i,\nu}$ such
that $g^{-1}\gm_{i+1} g\in I_{P_{i+1},\nu}$.

\smallskip

\noindent{\bf Step 4.} Recall that the projection $p_i:P_{i+1}\to P_i$, viewed as a morphism of algebraic varieties,
has a section $s$. Indeed, the isomorphism $P_{\psi}\isom M_{\psi}\times U_{\psi}$ from Section~\re{strata}(c) induces an isomorphism
$P_i\isom M_{\psi}\times (U_0/U_i)$. Choose an ordering of the all roots of $G$ lying in
$\Lie U_0/\Lie U_i$. Then the map $(x_{\al})_{\al}\mapsto\prod_{\al}x_{\al}$ defines an isomorphism
$\prod_{\al} U_{\al}\isom  U_0/U_i$, where $U_{\al}$ is the root
space of $\al$. We define $s$ to be the composition
\[
P_i\isom M_{\psi}\times (U_0/U_i)\isom M_{\psi}\times\prod_{\al}
U_{\al}\hra M_{\psi}\times (U_0/U_{i+1})\isom P_{i+1}.
\]

By construction, we have $s(P^{\sc}_i)\subseteq P^{\sc}_{i+1}$, so using $s$, we identify $\wt{\Fl}'_{P^{\sc}_{i+1},\nu,\gm_{i+1}}$ with
the space of pairs $(g,u)$, where $g\in L(P^{\sc}_i)$
and $u\in L(\ov{U}_{i})/I_{\ov{U}_i,\nu}$, satisfying
\begin{equation} \label{Eq:cond}
(s(g)u)^{-1}\gm_{i+1} (s(g)u)\in I_{P_{i+1},\nu}.
\end{equation}
Moreover, equation \form{cond} implies that $g^{-1}\gm_i g\in I_{P_i,\nu}$, thus $g\in \wt{\Fl}_{P^{\sc}_{i},\nu,\gm_i}\subseteq L(P^{\sc}_i)$.
\smallskip

\noindent{\bf Step 5.} For each $g\in\wt{\Fl}_{P^{\sc}_{i},\nu,\gm_i}$, we set $\wt{g}:=s(g)^{-1}\gm_{i+1}s(g)\in L(P^{\sc}_{i+1})$. Then we have $p_i(\wt{g})=g^{-1}\gm_i g\in I_{P_i,\nu}$, so there exists a unique $u_g\in L(\ov{U}_i)$ such
that $\wt{g}=u^{-1}_g s(g^{-1}\gm_{i}g)$. Then we have an equality
\[
(s(g)u)^{-1}\gm_{i+1}(s(g)u)=u^{-1}\wt{g}u=
u^{-1}(\wt{g}u\wt{g}^{-1})u^{-1}_g s(g^{-1}\gm_{i}g).
\]

\smallskip

Let $\wt{m}\in I_{M_{\psi},\nu}\subseteq L(M_{\psi})$ be the image of
$g^{-1}\gm_i g\in I_{P_i,\nu}$. Since $\ov{U}_{i}$ lies in the center
of $U_0/U_{i+1}$, we have $\wt{g}u\wt{g}^{-1}=\wt{m}u\wt{m}^{-1}$.
Moreover, since $g\in\wt{\Fl}_{P^{\sc}_{i},\nu,\gm_i}$, we get that
$g^{-1}\gm_i g\in I_{P_i,\nu}$. Hence by our construction of $s$ we have $s(g^{-1}\gm_i g)\in I_{P_{i+1},\nu}$, thus
our condition \form{cond} can be rewritten as
\[
u^{-1}(\wt{m}u\wt{m}^{-1})\in u_g I_{\ov{U}_i,\nu}.
\]

\smallskip

\noindent{\bf Step 6.} Since $\ov{U}_i$ is abelian, we have a canonical
isomorphism $\ov{U}_i\isom \Lie\ov{U}_i$. Therefore each $u_g\in L(\ov{U}_i)$ gives rise to an element $n_g\in \Lie L(\ov{U}_i)$,
and $\wt{\Fl}'_{P^{\sc}_{i+1},\nu,\gm_{i+1}}$ is identified with the moduli space of
pairs $(g,n)$, consisting of  $g\in\wt{\Fl}_{P^{\sc}_{i},\nu,\gm_i}$ and $n\in
\Lie L(\ov{U}_{i})/\Lie I_{\ov{U}_i,\nu}$ such that
\begin{equation} \label{Eq:cond'}
(\Ad\wt{m}-1)(n)\in n_g+\Lie I_{\ov{U}_i,\nu}.
\end{equation}

\vskip 4truept

\noindent{\bf Step 7.} Since $\gm\in M_{\psi}(K)\subseteq G(K)$ is regular semisimple, the operator $\Ad\gm-1$
is invertible on $\Lie\ov{U}_{i}(K)$, and we set
$d:=\val\det(\Ad\gm-1,\Lie\ov{U}_{i}(K))$. Since each $\wt{m}$ is
an $M_{\psi}(K)$-conjugate of $\gm$, we conclude that the
valuation of determinant of $\Ad\wt{m}-1$ on $\Lie\ov{U}_{i}(K)$
is $d$, thus the linear transformation of $\Lie L(\ov{U}_{i})/
\Lie{I}_{\ov{U}_i,\nu}$, induced by $\Ad\wt{m}-1$, has a kernel of dimension
$d$. Thus equation \form{cond'} implies that
$\wt{\Fl}'_{P^{\sc}_{i+1},\nu,\gm_{i+1}}$ is an affine subbundle of
$\wt{\Fl}_{P^{\sc}_{i},\nu,\gm_i}\times(\Lie L(\ov{U}_{i})/\Lie {I}_{\ov{U}_i,\nu})$
of dimension $d$.
\end{proof}

\begin{Prop} \label{P:ind}
Let $\ov{w}\in\wt{W}^{\C{C}}$ be an admissible tuple,
$\psi\in\Psi$, $\nu:=\ov{w}_{\psi}\in\wt{W}_{\psi}$, and let
$Z_{\nu}\subseteq \Fl$ as in Section~\re{strata}(a). Then exists $m\in\B{N}$
such that if $\ov{w}$ is $m$-regular, then

\smallskip

(a) the reduced intersections $Z_{\nu}\cap \Fl_{\gm}^{\leq\ov{w}}$ and
$Z_{\nu}\cap\left(\bigcap_{\psi\in C}\Fl_{\gm}^{\leq_{C}w_{C}}\right)$ are equal;

\smallskip

(b) the isomorphism $Z_{\nu}^{T_{\psi}}\isom \Fl_{M^{\sc}_{\psi}}$
from Sections~\re{strata}(e),(f) induces an isomorphism between the reduced intersection $Z^{T_{\psi}}_{\nu}\cap
\Fl_{\gm}^{\leq\ov{w}}$ and $\Fl_{M^{\sc}_{\psi},\gm}^{\leq
\ov{w}^{\psi}}$ (see Section~\ref{N:kreg}(e));

\smallskip

(c) we have an inclusion of sets $p_{\nu}^{-1}(p_{\nu}(\Fl^{\leq\ov{w}}\cap Z_{\nu,\gm}))
\subseteq \Fl^{\leq\ov{w}}$.
\end{Prop}

\begin{proof}
(a) Let $Y\subseteq \Fl_{\gm}$  be a closed subscheme of finite type
such that $\Fl_{\gm}=\La_{\gm}(Y)$ (see Section~\re{affsprfib}(b)). Then using, for
example, \rco{adm}(d),(e), there exists a finite stratification
$Y=\bigcup_{j} Y_{j}$ such that for every $j$ there exists a tuple
$\ov{u}\in\wt{W}^{\C{C}}$ such that $Y_j\subseteq L(U_C)u_C$ for
each $C\in\C{C}$.

Thus it is enough to show that $Z_{\nu}\cap
(\La_{\gm}(Y_j))^{\leq\ov{w}}$ equals
$Z_{\nu}\cap\left(\bigcap_{C\owns \psi}\La_{\gm}(Y_j)^{\leq_C w_C}\right)$ for
each $j$. In other words, we have to show that for every
$\mu\in\La_{\gm}\subseteq\La$ and $y\in Y_j$ such that $\mu(y)\in \Fl^{\leq_{C}
w_C}$ for every $C\owns \psi$ and $\mu(y)\in Z_{\nu}$, we have
$\mu(y)\in \Fl^{\leq_{C} w_C}$ for every $C\in\C{C}$.

Let $\ov{u}$ be a tuple of elements of $\wt{W}$ such that $y\in
L(U_C)u_C$ for every $C$. Then $y\in \Fl^{\leq\ov{u}}$ and it
follows from \rco{adm}(d) that $\ov{u}$ is admissible. By the
assumption, for every $C\owns\psi$ we have $\mu u_C\leq_C w_C$ and
also $(\mu u)_{\psi}=\nu=w_{\psi}$.

By \rl{order1}(a), this implies that $\mu u_C\leq_{C^{\psi}} w_C$ for every $C\owns\psi$.
Hence it follows from \rl{order2} that if $\ov{w}$ is sufficiently regular, then
 $\mu u_C\leq_{C} w_C$ for every $C\in\C{C}$, thus $\mu(y)\in \bigcap_{C\in\C{C}}\Fl^{\leq_{C} w_C}$, as claimed.

\smallskip

(b) By part~(a), the reduced intersections $Z^{T_{\psi}}_{\nu}\cap
\Fl_{\gm}^{\leq\ov{w}}$ and $Z^{T_{\psi}}_{\nu}\cap
\left(\bigcap_{C\owns\psi}\Fl_{\gm}^{\leq_{C}w_{C}}\right)$ are equal. Therefore it
suffices to show that the isomorphism
$Z_{\nu}^{T_{\psi}}\isom \Fl_{M^{\sc}_{\psi}}$ induces an isomorphism between the reduced intersection
$Z^{T_{\psi}}_{\nu}\cap \Fl^{\leq_{C}w_{C}}$ and
$\Fl^{\leq_{C^{\psi}}w_{C^{\psi}}}_{M^{\sc}_{\psi}}$ for all $C\owns\psi$. Since
$\Fl^{\leq_{C} w_C}$ is a closed $L(U_C)$-invariant ind-subscheme
of $\Fl$, we conclude that $Z^{T_{\psi}}_{\nu}\cap
\Fl^{\leq_{C}w_{C}}$ corresponds to a closed $L(U_C)\cap
L(M^{\sc}_{\psi})=L(U_{C^{\psi}})$-invariant ind-subscheme of
$\Fl_{M^{\sc}_{\psi}}$. Using \rp{semiinf}, the question is
equivalent to the assertion that if $w'\in \wt{W}^{\psi}$ then
$w'w_{\psi}\leq_C w^{\psi}w_{\psi}$ if and only if
$w'\leq_{C^{\psi}} w_{\psi}$. But this was shown in \rl{order1}(b).

\smallskip

(c) By part~(a), it is enough to show that
 $p_{\nu}^{-1}(p_{\nu}(\Fl^{\leq_C w_C}\cap
Z_{\nu,\gm}))\subseteq \Fl^{\leq_C w_C}$ for each $C\owns\psi$. Since
every fiber of $p_{\nu}$ lies in a single $L(U_{\psi})$-orbit, the
assertion follows from the inclusion $U_{\psi}\subseteq U_C$.
\end{proof}



\subsection{Finiteness of homology}

\begin{Emp} \label{E:homology}
{\bf Homology.} We fix a prime number $\ell$ different from the characteristic of $k$.

\smallskip

(a) For a scheme $Y$ of finite type over $k$ and $\C{F}\in D_c^b(Y,\qlbar)$
one can form the homology groups $H_i(Y,\C{F}):=(H^i(Y,\C{F}))^*$. We also set $H_i(Y):=H_i(Y,\qlbar)$.

\smallskip

(b) A closed embedding $\iota:X\hra Y$ induces a morphism
\[
\iota^*:H^i(Y,\C{F})\to H^i(Y,\iota_*\iota^*\C{F})=H^i(X,\iota^*\C{F})=H^i(X,\C{F}|_X),
\]
hence a morphism $\iota_*:H_i(X,\C{F}|_X)\to H_i(Y,\C{F})$.

\smallskip

(c) By part~(b), a closed embedding  $\iota:X\hra Y$ induces a morphism $\iota_*:H_i(X)\to H_i(Y)$.
Therefore for every ind-scheme $Y=\colim_i Y_i$ over $k$ one can form a homology
$H_i(Y):=\colim_i  H_i(Y_i)$.
\end{Emp}

The main goal of this section is to show the following finiteness
property of homology of affine Springer fibers:

\begin{Prop} \label{P:bound}
In the situation of Section~\re{affsprfib}(d), there exists an integer $r$ such
that for every  tuple $\ov{x}\in\B{Z}^{\Psi}$ and every $\psi\in
\Psi$, we have an equality of kernels
\begin{equation} \label{Eq:ker}
\Ker\left(H_i(\Fl^{\leq'\ov{x}}_{\gm})\to
H_i(\Fl^{\leq'\ov{x}+r\ov{e}_{\psi}}_{\gm})\right)=\Ker
\left(H_i(\Fl^{\leq'\ov{x}}_{\gm})\to
H_i(\Fl^{\leq'\ov{x}+(r+1)\ov{e}_{\psi}}_{\gm})\right).
\end{equation}
\end{Prop}

In order to prove this we need to introduce certain notation,
generalizing \cite[Sections~A.4.2 and 3.1.2]{BV}.

\begin{Emp} \label{E:filt}
{\bf Filtrations.} Let $\Gm$ be an ordered monoid, that is, a
monoid and a partially ordered set such that
$\si\tau\leq\si'\tau'$ for each $\si\leq\si'$ and $\tau\leq\tau'$.

\smallskip

(a) By a {\em $\Gm$-filtered set} (or a set with a {\em
$\Gm$-filtration}), we mean a set $X$ together with collections of
subsets $\{X_{\si}\}_{\si\in\Gm}$ such that $X_{\si}\subseteq
X_{\tau}$ for all $\si\leq\tau$, and $X=\bigcup_{\si} X_{\si}$.

\smallskip

(b) By a $\Gm$-filtered group we mean a group $A$ with a
$\Gm$-filtration such that $1\in A_1$ and $A_{\si}\cdot
A_{\tau}\subseteq A_{\si\tau}$.

\smallskip

(c) Let $A$ be a $\Gm$-filtered group, and $X$ is a set equipped
with an $A$-action and a $\Gm$-filtration. We say that
$\Gm$-filtration on $X$ is {\em compatible} with a filtration of
$A$, if for every $\si,\tau\in\Gm$ we have
$A_{\si}(X_{\tau})\subseteq X_{\si\tau}$.

\smallskip

(d) In the situation of (c), we will say that the $\Gm$-filtration
on $X$ is {\em finitely generated over $A$}, if there exists a
finite subset $\Gm_0\subseteq\Gm$ such that
$\{X_{\si}\}_{\si\in\Gm}$ is {\em generated by
$\{X_{\si}\}_{\si\in\Gm_0}$}, that is, for every $\si\in\Gm$, we
have
$X_{\si}=\bigcup_{\{(\tau,\si)\in\Gm\times\Gm_0\,|\,\tau\si'=\si\}}
A_{\tau}(X_{\si'})$.
\end{Emp}

\begin{Emp} \label{E:rees}
{\bf Rees algebras and modules.} Let $L$ be a field, and assume
that we are in the situation of Section~\re{filt}.

\smallskip

(a) For a $\Gm$-filtered group $A$, the group algebra $L[A]$
is also equipped with a $\Gm$-filtration
$L[A]_{\si}:=\Span_L(A_{\si})$, and we denote by
$R(L[A]):=\bigoplus_{\si\in\Gm} L[A]_{\si}$ the corresponding Rees algebra. Note that
$R(L[A])$ is the monoid algebra of the monoid
$R(A):=\{(a,\si)\in A\times\Gm\,|\, a\in A_{\si}\}$.

\smallskip

(b) Let $X$ be scheme locally of finite type over $k$ equipped
with an action of $A$. Assume that $A$ is a $\Gm$-filtered group,
and that  $X$ is equipped with a $\Gm$-filtration compatible with
$\Gm$-filtration on $A$ and such that $X_{\si}\subseteq X$ is a
closed subscheme of finite type over $k$ for each $\si\in\Gm$.

\smallskip

(c) For every $A$-equivariant element $\C{F}\in D_c^b(X,\qlbar)$ we can form a $\Gm$-graded
$R(\qlbar[A])$-module $R(H_i(X,\C{F})):=\bigoplus_{\si}
H_i(X_{\si},\C{F}\,|_{X_{\si}})$ for every $i\in\B{Z}$.
Explicitly, the action of $a\in A_{\tau}$ on $X$ defines a closed
embedding $a:X_{\si}\hra X_{\tau\si}$, hence a homomorphism
$H_i(X_{\si},\C{F}\,|_{X_{\si}})\to
H_i(X_{\tau\si},\C{F}\,|_{X_{\tau\si}})$ (see Section~\re{homology}(b)).

\smallskip

In particular, we form a $\Gm$-graded $R(\qlbar[A])$-module
$R(H_i(X)):=R(H_i(X,\qlbar))$.
\end{Emp}

\begin{Lem} \label{L:good}
In the situation of Section~\re{rees}(b), assume that

\smallskip

$\bullet$ the group $A$ acts on the set of irreducible
components of $X$ with finite stabilizers;

\smallskip

$\bullet$ the filtration $\{X_{\si}\}_{\si}$ is finitely generated over $A$;

\smallskip

$\bullet$ that the Rees algebra $R(\qlbar[A])$ is Noetherian.

\smallskip

\noindent Then for every $A$-equivariant object $\C{F}\in D_c^b(X,\qlbar)$ and $i\in\B{Z}$, the $R(\qlbar[A])$-module $R(H_i(X,\C{F}))$ is finitely
generated.
\end{Lem}

\begin{proof}
The argument is identical to that of \cite[Lemma~3.1.3]{BV}, where the case
of $\Gm=\B{Z}_{\geq 0}$ is treated.
\end{proof}

\begin{Emp} \label{E:ex}
{\bf Example.} (a) Let $\Gm$ be an ordered monoid $\B{Z}_{\geq 0}^{\Psi}$, which we identify with a corresponding submonoid of
the group of quasi-admissible tuples in $\La$ via the correspondence of Section~\re{qadm}(b),(c).

\smallskip

(b) Let $\La'\subseteq\La$ be a subgroup. Consider a $\Gm$-filtration on $\La'$, where for every
$\ov{x}\in\Gm$ we set $\La'_{\ov{x}}:=\La'\cap V^{\leq\ov{x}}$,
where $V^{\leq\ov{x}}$ is defined in Section~\ref{N:kreg}(d). Then
$\{\La'_{\ov{x}}\}_{\ov{x}}$ is a $\Gm$-filtered semigroup.

\smallskip

(c) Note that $R(\La')=\{(\mu,\ov{x})\in \La'\times\B{Z}_{\geq
0}^{\Psi}\,|\,\langle \psi,\mu\rangle\leq x_{\psi}\text{ for every
}\psi\in\Psi\}$. Therefore by Gordan's lemma (see, for example,
\cite[Lemma~3.4, page~154]{Ew}), $R(\La')$ is a finitely generated commutative
monoid. Therefore the Rees algebra $R(\qlbar[\La'])=\qlbar[R(\La')]$ is a finitely generated commutative
algebra over $\qlbar$, hence it is Noetherian.

\smallskip

(d) We apply the construction of part~(b) to $\La':=\La_{\gm}$, and equip the ind-scheme $X=\Fl_{\gm}$ (resp. $X=\Gr_{\gm}$)
with a $\Gm$-filtration $\Fl^{\leq'\ov{x}}_{\gm}$ (resp.
$\Gr^{\leq\ov{x}}_{\gm}$). Then it follows from definitions that
this filtration is compatible with a $\Gm$-filtration on $\La_{\gm}$.
\end{Emp}

\begin{Lem} \label{L:fg}
The $\Gm$-filtrations $\{\Gr^{\leq\ov{x}}_{\gm}\}_{\ov{x}}$ on $\Gr_{\gm}$ and $\{\Fl^{\leq'\ov{x}}_{\gm}\}_{\ov{x}}$ on $\Fl_{\gm}$ are finitely
generated over $\La_{\gm}$.
\end{Lem}

\begin{proof}
Since the filtration $\{\Fl^{\leq'\ov{x}}_{\gm}\}_{\ov{x}}$ on $\Fl_{\gm}$ is defined to be the preimage
of the filtration $\{\Gr^{\leq\ov{x}}_{\gm}\}_{\ov{x}}$ on $\Gr_{\gm}$, it will suffice to show the
assertion for $\{\Gr^{\leq\ov{x}}_{\gm}\}_{\ov{x}}$.

Notice that for every $\La_{\gm}$-invariant subset of $X\subseteq
\Gr_{\gm}$, the $\Gm$-filtration on $\Gr_{\gm}$ induces a
$\Gm$-filtration on $X$. Moreover, if $\Gr_{\gm}$ is a finite union
$\bigcup_j X_j$ of $\La_{\gm}$-invariant subsets, then the filtration
on $\Gr_{\gm}$ is finitely generated if and only if the
corresponding filtration on each $X_j$ is finitely generated.

Recall that there exists a closed subscheme of finite type $Y\subseteq
\Gr_{\gm}$ such that $\Gr_{\gm}=\La_{\gm}(Y)$. Moreover, using
\rco{adm}(d), there exists a finite decomposition $Y=\bigcup_j Y_j$
such that each for each $j$ there exists a tuple $\ov{y}=\ov{y}_j$
such that $Y_j\subseteq L(U_C)y_C$ for all $C\in\C{C}$. Then $\Gr_{\gm}=\bigcup_j \La_{\gm}(Y_j)$, and it suffices to show
that the filtration $\{\La_{\gm}(Y_j)^{\leq\ov{x}}\}_{\ov{x}}$ on each
$\La_{\gm}(Y_j)$ is finitely generated over $\La_{\gm}$.

Note that for every $\ov{x}\in\Gm$ we have an equality
$\La_{\gm}(Y_j)^{\leq\ov{x}}=\La_{\gm}^{\leq\ov{x}-\ov{y}_j}(Y_j)$. Indeed, it follows from \cite[Proposition~3.1]{MV}
(or can be deduced from \rp{semiinf}) that for every $\mu\in\La_{\gm}$ and $z\in Y_j$ we have
$\mu z\in \La_{\gm}(Y_j)^{\leq\ov{x}}$ if and only if $\mu y_C\leq_C x_C$ for all $C\in\C{C}$. Hence
$\mu z\in \La_{\gm}(Y_j)^{\leq\ov{x}}$ if and only if $\mu\in\La_{\gm}^{\leq\ov{x}-\ov{y}_j}$ as claimed.

Therefore it is enough to show that the $\Gm$-filtration $\{(\La_{\gm})_{\ov{x}-\ov{y}_j}\}_{\ov{x}}$ is finitely generated over $\La_{\gm}$. Since
$R(\qlbar[\La_{\gm}])$ is a finitely generated $\qlbar$-algebra (by Section~\re{ex}(c)), the assertion follows.
\end{proof}

\begin{Cor} \label{C:Rees}
The Rees module $R(H_i(\Fl_{\gm}))$  is a finitely generated
$R(\qlbar[\La_{\gm}])$-module.
\end{Cor}

\begin{proof}
Since Rees algebra $R(\qlbar[\La_{\gm}])$ is Noetherian (see Section~\re{ex}(c)), the assertion follows from Lemmas~\ref{L:good} and \ref{L:fg}.
\end{proof}

Now we are ready to prove \rp{bound}.

\begin{Emp}
\begin{proof}[Proof of \rp{bound}]
Since $\Psi$ is finite, it will suffice to show the existence of
$r$ for a fixed $\psi$. For every $r\in\B{N}$, the embeddings
$\Fl^{\leq'\ov{x}}\hra \Fl^{\leq'\ov{x}+r\ov{e}_{\psi}}$ for all
$\ov{x}$, induce a homomorphism of
$R(\qlbar[\La_{\gm}])$-modules
$\iota_{r\ov{e}_{\psi}}:R(H_i(\Fl_{\gm}))\to R(H_i(\Fl_{\gm}))$,
and \rp{bound} asserts that $\Ker\iota_{r\ov{e}_{\psi}}=\Ker\iota_{(r+1)\ov{e}_{\psi}}$
for some $r$.

Since $\{\Ker\iota_{r\ov{e}_{\psi}}\}_r$ is an increasing sequence
of $R(\qlbar[\La_{\gm}])$-submodules of $R(H_i(\Fl_{\gm}))$, the Rees algebra
$R(\qlbar[\La_{\gm}])$ is Noetherian (by Section~\re{ex}(c)), while
$R(H_i(\Fl_{\gm}))$ is finitely generated (by \rco{Rees}), this
sequence stabilizes.
\end{proof}
\end{Emp}


The following lemma will be used in the proof of \rt{inj'}.

\begin{Lem} \label{L:trunc}
There exists $m\in\B{N}$ such that for every $m$-regular tuple
$\ov{x}\in\B{Z}^{\Psi}$ and every $\psi\in\Psi$ such that
$\check{\psi}\notin(\La_{\gm})_{\B{Q}}$, we have
$\Fl_{\gm}^{\leq'\ov{x}}=\Fl_{\gm}^{\leq'\ov{x}+\ov{e}_{\psi}}$.
\end{Lem}

\begin{proof}
Our argument is similar to that of \rl{fg}. It is enough to show the equality $\Gr_{\gm}^{\leq\ov{x}}=\Gr_{\gm}^{\leq\ov{x}+\ov{e}_{\psi}}$.
Let $Y, Y_j$ and $\ov{y}_j$ be as in the proof of \rl{fg}, and  choose $m\in\B{N}$ such that for
every $m$-regular $\ov{x}$, the tuples $\ov{x}-\ov{y}_j+\ov{e}_{\psi}$
is regular. We claim that this $m$ satisfies the required property.

\smallskip

It suffices to show that $\La_{\gm}(Y_j)^{\leq\ov{x}}=\La_{\gm}(Y_j)^{\leq\ov{x}+\ov{e}_{\psi}}$
for each $j$. For this it suffices to show that $\La_{\gm}^{\leq \ov{x}-\ov{y}_j}=
\La_{\gm}^{\leq \ov{x}-\ov{y}_j+\ov{e}_{\psi}}$. In other words, we have to show that every  $\mu\in \La^{\leq \ov{x}-\ov{y}_j+\ov{e}_{\psi}}\sm
\La^{\leq \ov{x}-\ov{y}_j}$ does not belong to $\La_{\gm}$.

We are going to deduce the assertion from \rl{k-reg}(b).
Since $\check{\psi}\notin(\La_{\gm})_{\B{Q}}$, it follows from Section~\re{remgood}(d) that there exists a root
$\al\in\Phi$ such that $\al\in (\La_{\gm})^{\perp}$ and $\langle\al,\check{\psi}\rangle>0$. Since $\lan\psi,\mu\ran=(\ov{x}-\ov{y}_j+\ov{e}_{\psi})_{\psi}$, and the tuple $\ov{x}-\ov{y}_j+\ov{e}_{\psi}$ is regular by assumption, we conclude from  \rl{k-reg}(b) that $\langle\al,\mu\rangle>0$.
Therefore  $\mu\notin\La_{\gm}$, because $\al\in(\La_{\gm})^{\perp}$.
\end{proof}

%

\section{Proof of \rt{inj'}}

\subsection{Localization theorem for equivariant cohomology} \label{S:eqcoh}
In this section we will review basic facts about equivariant cohomology (with compact support), including a version of a localization theorem.

\begin{Emp} \label{E:cohart}
{\bf Total cohomology of Artin stacks.} For an Artin stack $\C{X}$ of finite type over $k$ and
$\C{F}\in D^b_c(\C{X},\qlbar)$, we denote by $H^{\bullet}(\C{X},\C{F}):=\bigoplus_i H^i(\C{X},\C{F})$ its total
cohomology, and set $H^{\bullet}(\C{X}):= H^{\bullet}(\C{X},\qlbar)$.

\smallskip

(a) Note that $H^{\bullet}(\C{X})=\Ext^{\bullet}_{\C{X}}(\qlbar,\qlbar)$ is a graded $\qlbar$-algebra, and identification
$\C{F}[\bullet]=\qlbar[\bullet]\otimes_{\qlbar}\C{F}$ give to $H^{\bullet}(\C{X},\C{F})$ a natural structure of a graded
$H^{\bullet}(\C{X})$-module.

\smallskip

(b) Every morphism
$\C{F}_1\to\C{F}_2$ in $D^b_c(\C{X},\qlbar)$ gives rise to a homomorphism
$H^{\bullet}(\C{X},\C{F}_1)\to H^{\bullet}(\C{X},\C{F}_2)$ of graded $H^{\bullet}(\C{X})$-modules.

\smallskip

(c) For every homomorphism $f:\C{X}\to\C{Y}$ of Artin stacks of finite type over $k$,
the pullback $f^*:H^{\bullet}(\C{Y})\to H^{\bullet}(\C{X})$ is a homomorphism of graded $\qlbar$-algebras. Moreover,
for every $\C{F}\in D^b_c(\C{Y},\qlbar)$ the pullback $f^*$ gives rise to a homomorphism
\[
H^{\bullet}(\C{X})\otimes_{H^{\bullet}(\C{Y})}H^{\bullet}(\C{Y},\C{F})\to H^{\bullet}(\C{X},f^*\C{F})
\]
of graded $H^{\bullet}(\C{X})$-modules.

\end{Emp}

\begin{Emp} \label{E:eqcoh}
{\bf Equivariant cohomology} (compare \cite{BL, GKM, Ac, AF}). Let $G$ be an algebraic group over $k$, let $X$ be a separated scheme of finite type over $k$ equipped
with a $G$-action, set $\pt:=\Spec k$, let $[\pt/G]$ be the classifying stack of $G$, and let $\pr_X:[X/G]\to[\pt/G]$ be the natural projection.

\smallskip

(a) For every $\C{F}\in D^b_c([X/G],\qlbar)$, we define its equivariant cohomology
\[
H^{\bullet}_{G}(X,\C{F}):=H^{\bullet}([X/G],\C{F}))=H^{\bullet}([\pt/G],R(\pr_X)_*(\C{F})),
\]
equivariant cohomology with compact support
\[
H^{\bullet}_{c,G}(X,\C{F}):=H^{\bullet}([\pt/G],R(\pr_X)_!(\C{F})),
\]
and set $H^{\bullet}_{G}(X):=H^{\bullet}_{G}(X,\qlbar)$ and $H^{\bullet}_{c,G}(X):=H^{\bullet}_{c,G}(X,\qlbar)$.

\smallskip

(b) By Section~\re{cohart}(a), $H^{\bullet}_G(\pt)$ is a graded $\qlbar$-algebra, while both $H^{\bullet}_{G}(X,\C{F})$ and $H^{\bullet}_{c,G}(X,\C{F})$ have natural structures of graded $H^{\bullet}_G(\pt)$-modules.

\smallskip

(c) Note that $H^{\bullet}_G(X)=\Ext^{\bullet}_{[X/G]}(\qlbar,\qlbar)$ is a graded $\qlbar$-algebra, hence both $H^{\bullet}_{G}(X,\C{F})$ and $H^{\bullet}_{c,G}(X,\C{F})$ have natural structures of graded $H^{\bullet}_G(X)$-modules.

\smallskip

(d) Note that the structures of
$H^{\bullet}_{G}(X,\C{F})$ and $H^{\bullet}_{c,G}(X,\C{F})$ of graded $H^{\bullet}_G(\pt)$-modules from part~(b), are obtained from
structures of graded $H^{\bullet}_G(X)$-modules from part~(c) by the homomorphism
\[
(p_X)^*:H^{\bullet}_G(\pt)=H^{\bullet}([\pt/G])\to H^{\bullet}([X/G])=H^{\bullet}_G(X)
\]
of graded $\qlbar$-algebras from Section~\re{cohart}(c).

\end{Emp}

\begin{Emp} \label{E:propeq}
{\bf Simple properties.} Let $G$, $X$ and $\C{F}$ be as in Section~\re{eqcoh}.

\smallskip

(a) Using Section~\re{cohart}(b), for each closed $G$-invariant subscheme $Z\subseteq X$ the long exact
sequence for cohomology with compact support naturally upgrades to an exact sequence
\[
H^{\bullet}_{c,G}(Z)[-1]\overset{\dt}{\to} H^{\bullet}_{c,G}(X\sm Z)\to H^{\bullet}_{c,G}(X)\to H^{\bullet}_{c,G}(Z)\overset{\dt}{\to}  H^{\bullet}_{c,G}(X\sm Z)[1]
\]
of graded $H^{\bullet}_G(\pt)$-modules, functorial in $(X,Z)$.
\smallskip

(b) If $G$ acts trivially on $X$, then we have canonical isomorphism
\[
H^{\bullet}_{c,G}(X)\simeq   H^{\bullet}_G(\pt)\otimes_{\qlbar} H_c^{\bullet}(X)
\] of
graded $H^{\bullet}_G(\pt)$-modules, functorial in $X$. Indeed, since $[X/G]\simeq X\times[\pt/G]$, the assertion follows from
K\"unneth formula. Alternatively, choose a compactification
$j:X\hra \ov{X}$ of $X$, and apply \cite[Proposition~6.7.5]{Ac} for $H^{\bullet}_{\{1\}\times G}(\ov{X}\times\pt,(j_!\qlbar)\boxtimes\qlbar)$.

\smallskip

(c) Using observation of Section~\re{cohart}(c) applied to the projection $\pi:\pt\to[\pt/G]$ and an object $R(p_X)_!(\C{F})\in D^b_c([\pt/G],\qlbar)$, we have a homomorphism $\pi^*:H^{\bullet}_G(\pt)\to H^{\bullet}(\pt)=\qlbar$ of graded $\qlbar$-algebras and a homomorphism
\begin{equation} \label{Eq:tensor}
\qlbar\otimes_{H^{\bullet}_G(\pt)} H_{c,G}^{\bullet}(X,\C{F})\to H_{c}^{\bullet}(X,\C{F})
\end{equation}
of graded vector spaces (compare \cite[equation~(6.7.2)]{Ac}).

Moreover, if
$H^{\bullet}_{c,S}(X,\C{F})$ is a free graded $H^{\bullet}_S(\pt)$-module, then morphism \form{tensor} is an isomorphism.
Indeed, as in the proof of \cite[Lemma~6.7.4]{Ac}, one first reduces to the case $X=\pt$ in which case the assertion follows from
\cite[Lemma~6.7.3]{Ac}.
\end{Emp}

\begin{Emp} \label{E:locthm}
{\bf Localization theorem} (compare \cite{GKM, AF}). Let $S$ be an algebraic torus acting on a separated scheme $X$ of finite type over $k$.

\smallskip

(a) Recall that graded $\qlbar$-algebra $H^{\bullet}_S(\pt)$ is canonically isomorphic with the symmetric algebra
$\on{Sym}^{\bullet}_{\qlbar}(X^*(S)\otimes_{\B{Z}}\qlbar(-1)[-2])$, where $X^*(S)$ denote the group of characters of $S$, while $[-2]$ indicates that the vector space $X^*(S)\otimes_{\B{Z}}\qlbar(-1)$ is placed in degree $2$
(see, for example, \cite[Theorem~6.7.7]{Ac}).

We fix an isomorphism of $\qlbar$-vector spaces $\qlbar(-1)\simeq\qlbar$, thus
we can view $X^*(S)$ as a subset of $\on{Sym}^{\bullet}_{\qlbar}(X^*(S)\otimes_{\B{Z}}\qlbar)\simeq\on{Sym}^{\bullet}_{\qlbar}(X^*(S)\otimes_{\B{Z}}\qlbar(-1))\simeq H^{\bullet}_S(\pt)$.

\smallskip

(b) By Section~\re{eqcoh}(b), both $H_{c,S}^{\bullet}(X,\C{F})$ and $H_{S}^{\bullet}(X,\C{F})$ are graded  $H^{\bullet}_S(\pt)$-modules for every $\C{F}\in D^b_c([X/S],\qlbar)$. We claim that if $X^S=\emptyset$, then there exists $\la\in X^*(S)\subseteq H^{\bullet}_S(\pt)$, which acts on each  $H_{c,S}^{\bullet}(X,\C{F})$ and $H_{S}^{\bullet}(X,\C{F})$ as zero.

Indeed, by a particular case of the localization theorem
(see, for example, \cite[Chapter 7, Theorem 1.1]{AF}) there exists $\la\in X^*(S)$ such that the image of $\la$ under the pullback $(p_X)^*:H^{\bullet}_S(\pt)\to H_S^{\bullet}(X)$ is zero, so the assertion follows by the observation of Section~\re{eqcoh}(d).

\smallskip

(c) The pullback $H^{\bullet}_{c,S}(X,\C{F})\to H^{\bullet}_{c,S}(X^S,\C{F}|_{X^S})$ induces an isomorphism of localizations
\[
(X^*(S))^{-1} H^{\bullet}_{c,S}(X,\C{F})\isom (X^*(S))^{-1} H^{\bullet}_{c,S}(X^S,\C{F}|_{X^S}).
\]
Indeed, by part~(b) we have $(X^*(S))^{-1} H^{\bullet}_{c,S}(X\sm X^S,\C{F})=0$, so the assertion follows
from the exact sequence of Section~\re{propeq}(a).

\smallskip

(d) If $H^{\bullet}_{c,S}(X,\C{F})$ is a free (or more generally torsion free) $H^{\bullet}_S(\pt)$-module, then the restriction
map $H^{\bullet}_{c,S}(X,\C{F})\to H^{\bullet}_{c,S}(X^S,\C{F}|_{X^S})$ is injective. Indeed, our assumption implies that
the canonical map $H^{\bullet}_{c,S}(X,\C{F})\to (X^*(S))^{-1} H^{\bullet}_{c,S}(X,\C{F})$ is injective, so the assertion follows from part~(c).
\end{Emp}

\subsection{Criterion of injectivity}

\begin{Emp} \label{E:bmhom}
{\bf Borel--Moore homology.} To every scheme $X$ of
finite type over $k$ one associates the Borel-Moore homology groups
$H_{i,BM}(X):=H^i_c(X,\qlbar)^*$. In particular, we have
$H_{i,BM}(X)=H_i(X)$, if $X$ is proper over $k$. Also for every
closed subscheme $Z\subseteq X$, we have a long exact sequence
\[
\to H_{i,BM}(Z)\to H_{i,BM}(X)\to H_{i,BM}(X\sm Z)\to
H_{i-1,BM}(Z)\to.
\]
\end{Emp}

\begin{Lem} \label{L:ker}
Let $X$ be a closed subscheme of $Y$, and let
$\iota:H_{i,BM}(X)\to H_{i,BM}(Y)$ be the natural map.

\smallskip

(a) The map $\iota$ is injective, if there exists a closed
subscheme $Z\subseteq X$ is such that
\[
\Ker\left(H_{i,BM}(Z)\to H_{i,BM}(X)\right)=\Ker\left(H_{i,BM}(Z)\to H_{i,BM}(X)\to H_{i,BM}(Y)\right),
\]
and the map $H_{i,BM}(X\sm Z)\to H_{i,BM}(Y\sm Z)$ is injective.

\smallskip

(b) The map $\iota$ is injective, if there exists a closed
subscheme $Z\subseteq Y$ containing $Y\sm X$ such that the natural
map $H_{i,BM}(Z\cap X)\to H_{i,BM}(Z)$ is injective.
\end{Lem}

\begin{proof}
Both assertions follow from a straightforward diagram chase.
Namely,  assertion (a) follows from the commutative diagram

\[
\CD
H_{i,BM}(Z)@>>> H_{i,BM}(X) @>>> H_{i,BM}(X\sm Z) \\
@.       @VVV  @VVV\\
 @. H_{i,BM}(Y) @>>> H_{i,BM}(Y\sm Z)
\endCD
\]
with an exact first row, while assertion (b) follows from the commutative
diagram with exact rows
\[
\CD
H_{i+1,BM}(Z\sm (Z\cap X))@>>> H_{i,BM}(Z\cap X) @>>> H_{i,BM}(Z) \\
@|       @VVV  @VVV\\
H_{i+1,BM}(Y\sm X )@>>> H_{i,BM}(X) @>>> H_{i,BM}(Y).
\endCD
\]
\end{proof}

\begin{Not} \label{N:fibr}
By an {\em vector fibration} we will mean a flat morphism, whose
geometric fibers are fibrations of vector spaces.
\end{Not}

The following result uses notation of Section~\ref{S:eqcoh}.

\begin{Lem} \label{L:equiv}
(a) Let $S$ be a torus, let $Y$ be an $S$-equivariant scheme of finite
type over $k$, and let $X\subseteq Y$ be a closed $S$-invariant subscheme. Assume
that

\smallskip

\quad (i) the restriction map $H^{\bullet}_c(Y^S)\to H^{\bullet}_c(X^S)$ is surjective and

\smallskip

\quad (ii) both $H^{\bullet}_{S,c}(X)$ and $H^{\bullet}_{S,c}(Y\sm X)$ are free graded
$H^{\bullet}_S(\pt)$-modules.

\smallskip

\noindent Then $H^{\bullet}_{S,c}(Y)$ is a free graded $H^{\bullet}_S(\pt)$-module, and the
restriction map $H^{\bullet}_c(Y)\to H^{\bullet}_c(X)$ is surjective.

\smallskip

(b) Assume that $Y$ has a finite $S$-invariant filtration
$\emptyset=Y_0\subseteq Y_1\subseteq\ldots \subseteq Y_n=Y$ by closed reduced
subschemes such that for each $j=1,\ldots,n-1$,

\smallskip

\quad (i) the restriction map $H^{\bullet}_c(Y_j^S)\to H^{\bullet}_c(Y_{j-1}^S)$ is surjective and

\smallskip

\quad (ii) there exists an $S$-equivariant vector fibration $\pi_j:Y_j\sm
Y_{j-1}\to (Y_j\sm Y_{j-1})^S$.

\smallskip

\noindent  Then $H^{\bullet}_{S,c}(Y)$ is a free graded $H^{\bullet}_S(\pt)$-module.
\end{Lem}

\begin{proof}
(a) By Section~\re{propeq}(a), we have a commutative diagram
\[
\CD
@. H^{\bullet}_{c,S}(X) @>\dt_1>> H^{\bullet}_{c,S}(Y\sm X)[1] \\
@.       @VVV  @VVV\\
H^{\bullet}_{c,S}(Y^S)@>>> H^{\bullet}_{c,S}(X^S) @>\dt_2>> H^{\bullet}_{c,S}(Y^S\sm
X^S)[1]
\endCD
\]
of graded $H^{\bullet}_S(\pt)$-modules with exact bottom row, where vertical arrows are induced by the inclusion $Y^S\hra Y$. By
Section~\re{propeq}(b), we have canonical isomorphisms
\[
H^{\bullet}_{S,c}(Y^S)\cong H^{\bullet}_S(\pt)\otimes_{\qlbar} H^{\bullet}_c(Y^S)\text{ and }
H^{\bullet}_{S,c}(X^S)\cong H^{\bullet}_S(\pt)\otimes_{\qlbar}H^{\bullet}_c(X^S)
\] of
$H^{\bullet}_S(\pt)$-modules. Hence, by assumption~(i), the map $H^{\bullet}_{c,S}(Y^S)\to
H^{\bullet}_{c,S}(X^S)$ is surjective, therefore the connecting
homomorphism $\dt_2$ is zero.

By assumption (ii) and the localization theorem (see Section~\re{locthm}(d)), the right vertical map is injective,
hence the connecting homomorphism $\dt_1$ is zero as well. Thus, by Section~\re{propeq}(a), we get a short exact
sequence
\[
0\to H^{\bullet}_{c,S}(Y\sm X)\to H^{\bullet}_{c,S}(Y)\to H^{\bullet}_{c,S}(X)\to 0,
\]
of graded $H^{\bullet}_S(\pt)$-modules,  hence $H^{\bullet}_{S,c}(Y)$ is a free graded
$H^{\bullet}_S(\pt)$-module by assumption (ii).  In this case, we have canonical isomorphisms
\[
H^{\bullet}_c(Y)\cong \qlbar\otimes_{H^{\bullet}_S(\pt)} H^{\bullet}_{c,S}(Y)\text{ and }H^{\bullet}_c(X)\cong
\qlbar\otimes_{H^{\bullet}_S(\pt)}H^{\bullet}_{c,S}(X)
\] of graded vector spaces (see Section~\re{propeq}(c)),
therefore surjectivity of the map $H^{\bullet}_{c}(Y)\to H^{\bullet}_{c}(X)$
follows from the surjectivity of $H^{\bullet}_{c,S}(Y)\to H^{\bullet}_{c,S}(X)$.

\smallskip

(b) We are going to show the assertion by
induction on $n$. Assume first that $n=1$. Since $\pi:=\pi_1:Y\to Y^S$ is an
$S$-equivariant vector fibration of some relative dimension $N$, we
conclude that $R\pi_!(\qlbar)\cong\qlbar[2N](N)$. Therefore
\[
H^{\bullet}_{c,S}(Y,\qlbar)\cong H^{\bullet}_{c,S}(Y^S,\qlbar)[2N](N)\cong
H^{\bullet}_S(\pt)\otimes_{\qlbar} H^{\bullet}_c(Y^S)[2N](N)
\]
(use Section~\re{propeq}(b)) is a free graded $H^{\bullet}_S(\pt)$-module.

Now assume that $n>1$, and set $X:=Y_{n-1}$. By the induction hypothesis, both $H^{\bullet}_{S,c}(X)$ and
$H^{\bullet}_{S,c}(Y\sm X)$ are free graded $H^{\bullet}_S(\pt)$-modules. Therefore by assumption (i), all assumptions of part~(a)
are satisfied, thus $H^{\bullet}_{S,c}(Y)$ is a free graded $H^{\bullet}_S(\pt)$-module.
\end{proof}

\begin{Lem} \label{L:fibr}
Let $Y$ be an ind-scheme of ind-finite type over $k$ equipped with an action of a torus
$S$, let $p:Y\to Y^S$ be an $S$-equivariant vector fibration such that its restriction $p|_{Y^S}$ is the identity,
and let $X\subseteq Y$ be a reduced locally closed ind-subscheme such that we have an inclusion of sets $p^{-1}(p(X))\subseteq X$.

\smallskip

Then $X$ is equal to the schematic preimage $p^{-1}(X^S)\subseteq Y$. In particular, $X$ is $S$-invariant, and $p$ induces an $S$-equivariant vector fibration $p_X:X\to X^S$ such that $p_X|_{X^S}$ is the identity.
\end{Lem}

\begin{proof}
Notice that since the inclusion $p^{-1}(p(X))\supseteq X$ always holds, we have an equality of sets $p^{-1}(p(X))=X$, thus the ind-subscheme
$X\subseteq Y$ is $S$-invariant.

Next we claim that we have an equality of sets $p(X)=X^S$. Indeed, $p|_{Y^S}$ is the identity, we get $p(X^S)=X^S$ and
$p(X)\subseteq p^{-1}(p(X))$. Therefore we have inclusions
\[
X^S=p(X^S)\subseteq p(X)\subseteq p^{-1}(p(X))\cap Y^S\subseteq X\cap Y^S=X^S.
\]
By the proven above, we have an equality of sets $p^{-1}(X^S)=p^{-1}(p(X))=X$, and from this the assertion follows: Indeed, since $X$ is reduced and $S$ is a torus, we conclude that $X^S$ is reduced. Since $p$ is a smooth, the schematic preimage $p^{-1}(X^S)$ is reduced, so the equality of
reduced ind-subschemes $p^{-1}(X^S)=X$ follows from the corresponding equality of the underlying sets.
\end{proof}

\begin{Cor} \label{C:crit}
Let $Z$ be an $S$-equivariant ind-scheme of ind-finite type over $k$,
$\{Z_{\nu}\}_{\nu\in\Xi}$ an $S$-invariant stratification of $Z$,
$Y\subseteq Z$ an $S$-invariant locally closed subscheme of finite
type over $k$, and $X\subseteq Y$ an $S$-invariant closed subscheme.

\smallskip

Assume that for each $\nu\in\Xi$,

\smallskip

\quad (a) the stratum $Z_{\nu}^S$ is an open and closed ind-subscheme of $Z^S$;

\smallskip

\quad (b) the map $H_{i,BM}(X\cap Z_{\nu}^S)\to H_{i,BM}(Y\cap
Z_{\nu}^S)$ is injective for all $i$;

\smallskip

\quad (c) there exists an $S$-equivariant vector fibration
$p_{\nu}:Y\cap Z_{\nu}\to Y\cap Z_{\nu}^S$ between reduced intersections such that
the restriction $p_{\nu}|_{Y\cap Z_{\nu}^S}$ is the identity, and we have an inclusion of sets
$p_{\nu}^{-1}(p_{\nu}(X\cap Z_{\nu}))\subseteq X$.

\smallskip

Then the map $H_{i,BM}(X)\to H_{i,BM}(Y)$ is injective for all
$i$.
\end{Cor}

\begin{proof}
We are going to apply the criterion of \rl{equiv}(a).

\smallskip

It follows from assumption~(a) that $Y^S$ (resp. $X^S$) is a disjoint union of the
$Y\cap Z_{\nu}^S$'s (resp. $X\cap Z_{\nu}^S$'s). This observation together
with assumption~(b) implies that the map $H_{i,BM}(X^S)\to H_{i,BM}(Y^S)$
is injective for all $i$, which by duality implies that the map
$H^{\bullet}_c(Y^S)\to H^{\bullet}_c(X^S)$ is surjective.

\smallskip

It thus suffices to show that both $H^{\bullet}_{S,c}(X)$ and $H^{\bullet}_{S,c}(Y\sm X)$
are free graded $H^{\bullet}(S)$-modules. Indeed, \rl{equiv}(a) then would imply that
the restriction map $H^{\bullet}_c(Y)\to H^{\bullet}_c(X)$ is surjective, from which our assertion
would follow by duality.

\smallskip

We are going to apply the criterion of \rl{equiv}(b):

\smallskip

By assumption~(a), the
disjoint union $Y^S=\coprod_{\nu}(Y\cap Z^S_{\nu})$ is of finite
type, hence the set $\Xi_0:=\{\nu\in\Xi\,|\,Y\cap
Z_{\nu}^S\neq\emptyset\}$ is finite. On the other hand, by assumption~(c), we have $\Xi_0=\{\nu\in
\Xi\,|\,Y\cap Z_{\nu}\neq\emptyset\}$. Define a standard partial order on $\Xi_0$
requiring that $\al\leq\beta$ if
and only if $Z_{\al}\subseteq\ov{Z}_{\beta}$. Denote the cardinality
of $\Xi_0$ by $n$, write $\Xi_0$ in the form
$\Xi_0=\{\nu_1,\ldots,\nu_n\}$ such that $\nu_j$ is a minimal
element of the set $\{\nu_j,\ldots,\nu_n\}$ for all
$j=1,\ldots,n$.

For each $j=1,\ldots, n$ we denote by $Y_j$ the reduced intersection
$Y\cap (\bigcup_{t=1}^j Z_{\nu_t})$. Then by  construction, each
$Y_j\subseteq Y$ is closed,  $Y_n=Y$, and $Y_j\sm Y_{j-1}=Y\cap
Z_{\nu_j}$. It suffices to show that the induced filtrations $X_j:=X\cap Y_j$
of $X$ and $(Y\sm X)_j:=Y_j\cap (Y\sm X)$ satisfy the assumptions
of \rl{equiv}(b).

\smallskip

Since $Y^S$ is a disjoint union of the $(Y_j\sm Y_{j-1})^S$'s (by assumption (a)), assumption (i) of \rl{equiv}(b) follows.
Next, since $Y_j\sm Y_{j-1}=Y\cap Z_{\nu_j}$, we get $X_j\sm X_{j-1}=X\cap Z_{\nu_j}$
and $(Y\sm X)_j\sm (Y\sm X)_{j-1}=(Y\sm X)\cap Z_{\nu_j}$. Hence 
it remains to construct $S$-equivariant vector fibrations $X\cap Z_{\nu_j}\to X\cap
Z^S_{\nu_j}$ and $(Y\sm X)\cap Z_{\nu_j}\to (Y\sm X)\cap
Z^S_{\nu_j}$. But both fibrations are induced from vector fibration $p_{\nu_j}:Y\cap
Z_{\nu_j}\to Y\cap Z^S_{\nu_j}$ from assumption~(c) using \rl{fibr}.
\end{proof}

\subsection{The proof} \label{S:proof}

Now we are ready to prove our main result (\rt{inj'}).

\begin{Thm} \label{T:main}
There exists $m\in\B{N}$ (depending on $\gm$) such that for all $m$-regular admissible
tuples $\ov{w}_1,\ldots,\ov{w}_n\in\wt{W}^{\C{C}}$, the natural
map $H_i(\bigcup_{j=1}^n \Fl^{\leq \ov{w}_j}_{\gm})\to H_i(\Fl_{\gm})$
is injective for all $i\in\B{Z}$.
\end{Thm}

\begin{proof}
Set $Z':=\bigcup_{j=1}^n \Fl^{\leq \ov{w}_j}\subseteq \Fl$. We want to
show that if each $\ov{w}_j$ is sufficiently regular, then the
natural map $H_i(Z'_{\gm})\to H_i(\Fl_{\gm})$ is injective for
all $i\in\B{Z}$. To make the proof more structural we will divide it into steps.

\smallskip

\noindent{\bf Step 1.} Let $\ov{x}_0\in\La^{\C{C}}$ be an admissible tuple constructed in
\rl{finite} and such that $\Fl^{\leq'\bar{x}_0}\subseteq \Fl^{\leq\ov{w}_1}$,
and let $\{\ov{x}_l\}_{l\geq 0}$ be a sequence of
admissible tuples from \rl{seq}. Moreover, it follows from
Lemmas~\ref{L:finite} and \ref{L:seq}, that each $\ov{x}_l$ is sufficiently
regular, if each $\ov{w}_j$ is sufficiently regular.

\smallskip

Notice that $\{\Fl^{\leq'\ov{x}_{l}}_{\gm}\}_{l\geq 0}$ form an
exhausting increasing union of closed subsets of $\Fl_{\gm}$, hence
it is enough to show that for every $l>0$ the map
\[
H_i(Z'_{\gm}\cup \Fl^{\leq'\ov{x}_{l-1}}_{\gm})\to
H_i(Z'_{\gm}\cup \Fl^{\leq'\ov{x}_{l}}_{\gm})
\]
is injective for all $l$. Using inclusion
\[
(Z'_{\gm}\cup \Fl^{\leq'\ov{x}_{l}}_{\gm})\sm (Z'_{\gm}\cup \Fl^{\leq'\ov{x}_{l-1}}_{\gm})\subseteq \Fl^{\leq'\ov{x}_{l}}_{\gm},
\]
we conclude from \rl{ker}(b) that it suffices to show
that the map
\[
H_i((Z'_{\gm}\cap \Fl^{\leq'\ov{x}_{l}}_{\gm})\cup
\Fl^{\leq'\ov{x}_{l-1}}_{\gm})\to H_i(\Fl^{\leq'\ov{x}_l}_{\gm})
\]
is injective. We set $\ov{x}:=\ov{x}_{l-1}$. Then
$\ov{x}_{l}=\ov{x}+\ov{e}_{\psi}$ for some $\psi\in\Psi$, and we want
to show that the map
\[
H_i((Z'_{\gm}\cap \Fl^{\leq'\ov{x}+\ov{e}_{\psi}}_{\gm})\cup
\Fl^{\leq'\ov{x}}_{\gm})\to H_i(\Fl^{\leq'\ov{x}+\ov{e}_{\psi}}_{\gm})
\]
is injective.

\smallskip

\noindent{\bf Step 2.} If $\check{\psi}\notin(\La_{\gm})_{\B{Q}}$, then we have an equality $\Fl^{\leq'\ov{x}}_{\gm}=\Fl^{\leq'\ov{x}+\ov{e}_{\psi}}_{\gm}$ (by \rl{trunc}), so the assertion is tautological.

\smallskip

From now on assume that $\check{\psi}\in(\La_{\gm})_{\B{Q}}$.
Let $r\in\B{N}$ be as in  \rp{bound}. Then, by \rl{ker}(a), it is
enough to show that the map
\[
H_{i,BM}([(Z'_{\gm}\cap \Fl^{\leq'\ov{x}+\ov{e}_{\psi}}_{\gm})\cup
\Fl^{\leq'\ov{x}}_{\gm}]\sm
 \Fl^{\leq'\ov{x}-r\ov{e}_{\psi}}_{\gm})\to
H_{i,BM}(\Fl^{\leq'\ov{x}+\ov{e}_{\psi}}_{\gm}\sm
 \Fl^{\leq'\ov{x}-r\ov{e}_{\psi}}_{\gm})
\] is injective.

\smallskip

\noindent{\bf Step 3.}  We are going to apply the criterion of \rco{crit} in the case
 $Z=\Fl$, $S=T_{\psi}$,
\[
X= [(Z'_{\gm}\cap \Fl^{\leq'\bar{x}+\ov{e}_{\psi}}_{\gm})\cup
\Fl^{\leq'\ov{x}}_{\gm}]\sm
 \Fl^{\leq'\ov{x}-r\ov{e}_{\psi}}_{\gm},
\]
$Y=\Fl^{\leq'\ov{x}+\ov{e}_{\psi}}_{\gm}\sm
\Fl^{\leq'\ov{x}-r\ov{e}_{\psi}}_{\gm}$, and
$\{Z_{\nu}\}_{\nu\in\wt{W}_{\psi}}$ is the stratification of  $\Fl$
by $L(P^{\sc}_{\psi})$-orbits, considered in Section~\re{strata}.

\smallskip

Since $X$ and $Y$ are locally closed subschemes of $Z$ of finite type over $k$, it remains to show that $X$ and $Y$ are $S$-invariant and properties (a)-(c) of \rco{crit} are satisfied. Property (a) was mentioned in Section~\re{strata}(g).

\smallskip

\noindent{\bf Step 4.} Next we are going to show property (b). Observe that for every
stratum $Z_{\nu}$ such that $Y\cap Z_{\nu}\neq\emptyset$, we have
$x_{\psi}-r<\pi(\nu)_{\psi}\leq x_{\psi}+1$. If
$\pi(\nu)_{\psi}\leq x_{\psi}$, then we have $X\cap Z_{\nu}=Y\cap
Z_{\nu}$, in which case property (b) clearly holds.

\smallskip

Assume now that $\pi(\nu)_{\psi}=x_{\psi}+1$. In this case, we have an equality
\begin{equation} \label{Eq:xznu}
X\cap Z_{\nu}=(Z'\cap \Fl^{\leq'\ov{x}+\ov{e}_{\psi}}\cap \Fl^{\leq_{\psi}\nu}) \cap Z_{\nu,\gm},
\end{equation}
and it is enough to show that the composition
\begin{equation} \label{Eq:inj}
H_{i,BM}(X\cap Z^{S}_{\nu})\to H_{i,BM}(Y\cap
Z^{S}_{\nu})\to H_{i,BM}(Z^{S}_{\nu,\gm})
\end{equation}
is injective.

\smallskip

By a combination of \rl{int}(a),(b) and \rco{adm}(f), the intersection
$Z'\cap \Fl^{\leq'\ov{x}+\ov{e}_{\psi}}\cap \Fl^{\leq_{\psi}\nu}$ decomposes
as a finite union $\bigcup_t \Fl^{\leq\ov{u}_t}$, where each
$\ov{u}_t$ is sufficiently regular. Therefore \form{xznu} implies that the reduced intersection $X\cap Z_{\nu}$ is of the
form $\bigcup_t (\Fl_{\gm}^{\leq\ov{u}_t}\cap Z_{\nu})$, where each $\ov{u}_t$
is sufficiently regular, and $(\ov{u}_t)_{\psi}=\nu$. Hence the reduced intersection
$X\cap Z^{S}_{\nu}$ is of the form $\bigcup_t \Fl_{M^{\sc}_{\psi},\gm}^{\leq\ov{u}^{\psi}_t}$ (by \rp{ind}(b)),
and each $\ov{u}^{\psi}_t\in\wt{W}^{\psi}$ is sufficiently regular (by \rl{ord1}(d)).

\smallskip

By induction on the semisimple rank of $G$, we can assume that \rt{main} holds for the Levi subgroup $M_{\psi}$.
Therefore the map
\[
H_{i,BM}(\bigcup_t \Fl_{M^{\sc}_{\psi},\gm}^{\leq\ov{u}^{\psi}_t})\to H_{i,BM}(\Fl_{M^{\sc}_{\psi},\gm})
\]
is injective, from which the injectivity of \form{inj} and hence property (b) follows.

\smallskip

\noindent{\bf Step 5.}
It remains to show $X$ and $Y$ are $S$-invariant and satisfy property (c). Recall that in \rp{strata} we
constructed an $S$-equivariant retraction $p_{\nu,\gm}:Z_{\nu,\gm}\to Z_{\nu,\gm}^{S}$,
which is a composition of affine bundles, hence a vector
fibration.

\smallskip

By \rl{fibr} it is enough to show that $p_{\nu}$ satisfies inclusions of sets
\begin{equation} \label{Eq:incl}
p_{\nu}^{-1}(p_{\nu}(Y\cap Z_{\nu}))\subseteq Y\cap Z_{\nu}\text{
and }p_{\nu}^{-1}(p_{\nu}(X\cap Z_{\nu}))\subseteq X\cap Z_{\nu}.
\end{equation}

We will show the inclusion \form{incl} for $X$, while the
inclusion for $Y$ is similar, but simpler. Observe that if $Y\cap
Z_{\nu}$ is non-empty, then the intersection $X\cap Z_{\nu}$ is either
$(\Fl^{\leq'\ov{x}}\cap \Fl^{\leq_{\psi}\nu}) \cap Z_{\nu,\gm}$ or
$(Z'\cap \Fl^{\leq'\ov{x}+\ov{e}_{\psi}}\cap \Fl^{\leq_{\psi}\nu})\cap
Z_{\nu,\gm}$.

In both cases, a combination of \rl{int} and \rco{adm}(f) implies that  $X\cap Z_{\nu}$ is of the
form $\bigcup_t (\Fl^{\leq\ov{u}_t}\cap Z_{\nu,\gm})$, where each $\ov{u}_t$
is sufficiently regular, and $(\ov{u}_t)_{\psi}=\nu$. Thus to show inclusion \form{incl} for $X$ it suffices to check inclusions of sets  $p_{\nu}^{-1}(p_{\nu}(\Fl^{\leq\ov{u}_t}\cap Z_{\nu,\gm}))\subseteq
\Fl^{\leq\ov{u}_t}$ for all $t$. But this follows from \rp{ind}(c).
\end{proof}






\end{document}